\documentclass[oneside,reqno]{amsart}

\usepackage{amsfonts}
\usepackage{amssymb}
\usepackage{amsmath}
\usepackage{enumerate}

\DeclareMathOperator{\Ann}{Ann}

\DeclareMathOperator{\ind}{Ind}
\DeclareMathOperator{\Supp}{Supp}

\DeclareMathOperator{\ffd}{D_{\mathrm{ff}}}

\newcommand{\bfd}{\mathbf{d}}
\newcommand{\mcP}{{\mathcal P}}
\newcommand{\bbF}{{\mathbb F}}
\newcommand{\bbN}{{\mathbb N}}
\newcommand{\bbR}{{\mathbb R}}

\newcommand{\grft}{B^{(n)}}

\newcommand{\gft}{{\mathbb F}_2}

\newcommand{\var}{\mathrm{var}}
\newcommand{\grobner}{Gr\"{o}bner}
\newcommand{\Bn}{B^{(n)}}

\newcommand{\ceil}[1]{\ensuremath{\left\lceil #1 \right\rceil}}
\newcommand{\bn}[2]{\binom{#1}{#2}}

\begin{document}
\title{On the Existence of Semi-regular Sequences}
\author{Timothy J. Hodges}
\email[Timothy Hodges]{timothy.hodges@uc.edu}
\address{University of Cincinnati, Cincinnati, OH 45221-0025, USA}
\thanks{Corresponding author: Timothy Hodges, Department of Mathematical Sciences, University of Cincinnati, Cincinnati, OH 45221-0025, USA, email:timothy.hodges@uc.edu}
\thanks{The first author was partially supported by a grant from the Charles P. Taft Research Center}
\thanks{The second author was partially supported by a scholarship from COLFUTURO/Foundation for the Future of Colombia and by the Maita Levine Fellowship, University of Cininnati, USA}
\author{Sergio D. Molina}
\email[Sergio Molina]{sdaladierm@hotmail.com}
\address{University of Cincinnati, Cincinnati, OH 45221-0025, USA}
\author{Jacob Schlather}
\email[Jacob Schlather]{jacob.schlather@gmail.com}
\address{21 Sartwell Ave. Apt. 2, Somerville, MA 02144, USA}


\begin{abstract}
Semi-regular sequences over $\bbF_2$ are sequences of homogeneous elements of the algebra $
B^{(n)}=\mathbb{F}_2[X_1,...,X_n]/(X_1^2,...,X_n^2)
$, which have as few relations between them as possible. They were introduced in order to assess the complexity of \grobner\ basis algorithms such as ${\bf F}_4, {\bf F}_5$ for the solution of polynomial equations. Despite the experimental evidence that semi-regular sequences are common, it was unknown whether there existed semi-regular sequences for all $n$, except in extremely trivial situations.
We prove some results on the existence and non-existence of semi-regular sequences.  In particular, we show that if an element of degree $d$ in $B^{(n)}$ is semi-regular, then we must have $n\leq 3d$. Also, we show that if $d=2^t$ and $n=3d$ there exits a semi-regular element of degree $d$ establishing that the bound is sharp for infinitely many $n$. Finally,  we generalize the result of non-existence of semi-regular elements to the case of sequences of a fixed length $m$.
\end{abstract}
\maketitle

\newtheorem{thm}{Theorem}
\newtheorem{lem}[thm]{Lemma}
\newtheorem{prop}[thm]{Proposition}
\newtheorem{cor}[thm]{Corollary}
\newtheorem{Notation}[thm]{Notation}
\theoremstyle{definition}
\newtheorem{defn}[thm]{Definition}
\newtheorem{example}{Example}
\renewcommand{\thedefn}{}
\newtheorem{Conjecture}{Conjecture}

\theoremstyle{remark}
\newtheorem{notation}{Notation}
\renewcommand{\thenotation}{}
\newtheorem{question}{Question}

\numberwithin{thm}{section}
\numberwithin{equation}{section}

\section{Introduction}

Semi-regular sequences over $\bbF_2$ are sequences of homogeneous elements of the algebra 
$$
B^{(n)}=\mathbb{F}_2[X_1,...,X_n]/(X_1^2,...,X_n^2)
$$
which have as few relations between them as possible. They were introduced in \cite{b,bfs, bfs2,bfsy} in order to assess the complexity of \grobner\ basis algorithms such as ${\bf F}_4, {\bf F}_5$ for the solution of polynomial equations.

Experimental evidence has shown  that  randomly generated sequences tend to be semi-regular \cite[Section 3]{bfs2}. On the other hand it has been observed than many sequences that arise in cryptography, such as those arising from the Hidden Field Equation cryptosystems, are not semi-regular. Despite the experimental evidence that semi-regular sequences are common, it was unknown whether there existed semi-regular sequences for all $n$, except in extremely trivial situations.

We prove here some results on the existence and non-existence of semi-regular sequences. We first look at the most elementary case, that of semi-regular elements (or sequences of length one). It was observed in \cite[Lemma 3.12]{vk}, that there are no quadratic semi-regular elements when $n>6$. On the other hand it is trivial that elements of degree $n$ and $n-1$ must be semi-regular. This raises the question: for which values of $n$ and $d$ do there exist semi-regular elements of degree $d$ in $B^{(n)}$? We prove here that if an element of degree $d\geq 2$ in $B^{(n)}$ is semi-regular, then we must have $n\leq 3d$. We go somewhere towards understanding the sharpness of this bound by determining  when the symmetric polynomials
$$
\sigma_d(x_{1},...,x_{n})=\sum_{1\leq j_1<\cdots<j_d\leq n}x_{{j_1}}\cdots x_{{j_d}}
$$
are semi-regular. In particular we show that if $d=2^t$ and $n=3d$ then $\sigma_{n,d}$ is semi-regular, thus establishing that the bound is sharp for infinitely many $n$.

Thus for any fixed $d$, there are no semi-regular elements in $\Bn$ for $n >3d$. We generalize this result to the case of sequences of a fixed length $m$, though in a predictably less precise fashion. Define the degree of a sequence $\lambda_1, \dots, \lambda_m$ to be the vector ${\bf d}= (\deg \lambda_1, \dots, \deg \lambda_m)$. We show that for all such vectors $\bf d$ there exists an $N$ such that for $n>N$, there are no semi-regular sequences of degree $\bf d$ in $\Bn$.

\section{Semi-Regularity over $\gft$}

Set $B^{(n)}=\mathbb{F}_2[X_1,...,X_n]/(X_1^2,...,X_n^2)$ and define $\Bn_k$ to be the subspace of homogeneous polynomials of degree exactly $k$. Then $\Bn= \bigoplus_{k=0}^n \Bn_k$ and $\Bn_i\Bn_j=\Bn_{i+j}$, so this gives $\Bn$ the structure of a strongly graded $\gft$-algebra. Denote the image of $X_i$ in $\Bn$ by $x_i$. 

\begin{defn} For a graded ring $B=\bigoplus_{k=0}^N B_k$, we define the index of $B$ to be $t$ if $B_{t-1}\neq0$ and $B_k=0$ for all $k\geq t$. We denote this number by $\ind(B)$. If $\lambda_1, \dots, \lambda_m$ is a set of homogeneous elements and $I = (\lambda_1, \dots, \lambda_m)$, then we define $\ind(\lambda_1, \dots, \lambda_m)= \ind(B/I)$. If $B$ is strongly graded (that is, $B_iB_j=B_{i+j}$ for all $i$ and $j$), then
$$
\ind(B/I) = \min \{ d \geq 0 \mid I \cap B_d = B_d\}
$$
\end{defn}

\begin{defn} Let $\lambda_1, \dots, \lambda_m$ be a sequence of homogeneous elements of $\Bn$ of positive degree. The sequence $\lambda_1, \dots, \lambda_m$ is {\em $D$-semi-regular} if for all $i = 1, 2, \dots, m$, if $\mu$ is homogeneous and
$$
\mu\lambda_i \in (\lambda_1, \dots, \lambda_{i-1}) \quad \text{and}\quad \deg(\mu)+\deg(\lambda_i)< D
$$
then $\mu \in (\lambda_1, \dots, \lambda_{i})$.
A sequence of homogeneous polynomials $\lambda_1, \dots, \lambda_m$ is {\em semi-regular} if it is $D$-semi-regular for $D= \ind(\lambda_1, \dots, \lambda_m)$
\end{defn}

Recall that the Hilbert function of a graded ring $B$ is the function $HF_B(k) = \dim B_k$ and the Hilbert series is the series $HS_B(z) = \sum_{k=0}^\infty (\dim B_k) z^k$. The Hilbert function and series for a graded ideal $I$ of $B$ are defined by $HF_I=HF_{B/I}$ and $HS_I(z) = HS_{B/I}(z)$ respectively. For any series $a(z)=\sum_ia_iz^i \in \bbR[[z]]$ , we define the index of $a(z)$, $\ind a(z)$, to be the first $t$ for which $a_t\leq 0$. If such a $t$ does not exist define $\ind a(z)=\infty$. For a series $\sum_ia_iz^i$, we denote by $\left[ \sum_ia_iz^i \right]_t$ the truncated series $\sum_{i=0}^{t-1} a_iz^i$ and by $\left[ \sum_ia_iz^i \right]$ the truncated series $\left[ \sum_ia_iz^i \right]_{\ind(a(z))}$.
 It was asserted in \cite{bfs} that a sequence $\lambda_1, \dots, \lambda_m$ is semi-regular if and only if 
$$
HS_{(\lambda_1, \dots, \lambda_m)}(z) =\left[ \frac{(1+z)^n}{\prod_{i=1}^m(1+z^{d_i})}\right]
$$
where $d_i = \deg \lambda_i$. As noted in \cite{Diem}, the proofs in that article are incomplete. We begin, therefore, by giving a complete proof.

For any ${\bf d} =(d_1, \dots, d_m) \in \bbN^m$, define
$$
T_{{\bf d},n}(z) =\frac{(1+z)^n}{\prod_{i=1}^m(1+z^{d_i})}
$$
and let $t_{{\bf d},n}(j)$ be the coefficient of $z^j$ in $T_{{\bf d},n}(z)$, so that $T_{{\bf d},n}(z)= \sum_{i=0}^\infty t_{{\bf d},n}(j) z^j$.

We begin with some observations on the way truncation behaves with respect to multiplication.

\begin{lem} \label{trunc} Let $u(z), v(z), w(z) \in \bbR[[z]]$. Then 
\begin{enumerate}
  \item $\left[u(z)v(z)\right]_D = \left[ \left[u(z)\right]_D \left[v(z)\right]_D \right]_D = \left[ u(z) \left[v(z)\right]_D \right]_D $
  \item $\left[v(z)\right]_D= \left[w(z)\right]_D \implies \left[u(z)v(z)\right]_D = \left[u(z)w(z)\right]_D$
\end{enumerate}  
\end{lem}

\begin{proof} (1) First note that for any $a(z), c(z) \in \bbR[[z]]$,
$$
[a(z) + c(z) z^D]_D = [a(z)]_D
$$
Define $u'(z), v'(z)$ by  $u(z)=[u(z)]_D+u'(z)z^D$ and $v(z)=[v(z)]_D+v'(z)z^D$. Then
\begin{align*}
\left[u(z)v(z)\right]_D &= \left[([u(z)]_D+ u'(z)z^D)( [v(z)]_D+ v'(z)z^D)\right]_D \\
	& = \left[[u(z)]_D [v(z)]_D+ z^D(u'(z)[v(z)]_D + v'(z)[u(z)]_D + u'(z)v'(z)z^D )\right]_D\\
	& = \left[ \left[u(z)\right]_D \left[v(z)\right]_D \right]_D
\end{align*}
So
$$
\left[ u(z) \left[v(z)\right]_D \right]_D = \left[ \left[u(z)\right]_D \left[v(z)\right]_D \right]_D =\left[u(z)v(z)\right]_D
$$

(2) If $\left[v(z)\right]_D= \left[w(z)\right]_D$, then 
$$
\left[u(z)v(z)\right]_D = \left[u(z)[v(z)]_D\right]_D = \left[u(z)[w(z)]_D\right]_D =\left[u(z)w(z)\right]_D
$$
\end{proof}

Since $\lambda^2=0$ for any homogeneous element $\lambda$ of positive degree, multiplication by $\lambda_i$ is a well-defined map  from $B/(\lambda_1, \dots, \lambda_{i})$ to $B/(\lambda_1, \dots, \lambda_{i-1})$ whose image is $(\lambda_1, \dots, \lambda_{i})/(\lambda_1, \dots, \lambda_{i-1})$. Let $\pi_i$ be the natural projection from $B/(\lambda_1, \dots, \lambda_{i-1})$ to $ B/(\lambda_1, \dots, \lambda_{i})$. Thus we have an exact sequence,
$$
B/(\lambda_1, \dots, \lambda_{i}) \overset{\lambda_i}{\longrightarrow} B/(\lambda_1, \dots, \lambda_{i-1}) \overset{\pi_i}{\longrightarrow} B/(\lambda_1, \dots, \lambda_{i})\to 0
$$
With this notation we can restate the definition of $D$-semi-regularity by saying that a sequence $\lambda_1, \dots, \lambda_m$ is {\em $D$-semi-regular} if and only if the sequence 
$$ 0 \to (B/(\lambda_1, \dots, \lambda_{i}))_{d-d_i} \to (B/(\lambda_1, \dots, \lambda_{i-1}))_{d} \to (B/(\lambda_1, \dots, \lambda_{i}))_{d}\to 0
$$
is exact for all $i=1, \dots, m$ and all $d <D$.

\begin{thm} \label{bfsthm}
Let $B = \grft$ and let $\lambda_1, \dots, \lambda_m$ be a sequence  of homogeneous elements of $B$ with $\lambda_i$ being of degree $d_i$. Set $I = (\lambda_1, \dots, \lambda_m)$ and ${\bf d} =(d_1, \dots, d_m)$.
\begin{enumerate}
  \item If  the sequence $\lambda_1, \dots, \lambda_m$ is $D$-semi-regular then 
  $$ [HS_I(z)]_D= \left[ T_{{\bf d},n}(z) \right]_D $$
and $HF_I(D) \geq t_{{\bf d},n}(D)$.
  \item If the sequence $\lambda_1, \dots, \lambda_m$ is $D$-semi-regular but not $(D+1)$-semi-regular, then $HF_I(D) > t_{{\bf d},n}(D)$.
  \item The sequence $\lambda_1, \dots, \lambda_m$ is semi-regular if and only if the Hilbert series of $I$ is given by
  $$ HS_I(z)= \left[ T_{{\bf d},n}(z) \right] $$
\item If the sequence $\lambda_1, \dots, \lambda_m$ is $D$-semi-regular, then so is $\lambda_{\sigma(1)}, \dots, \lambda_{\sigma(m)}$ for any permutation $\sigma$.
\end{enumerate}
\end{thm}

\begin{proof} Set ${\bf d}_i = (d_1, \dots, d_i)$ and denote $t_{{\bf d}_i,n}(d)$ by $t_{i}(d)$. Note that 
$$
(1+z^{d_j}) \sum_{d=0}^\infty t_j(d)z^d = \sum_{d=0}^\infty t_{j-1}(d)z^d
$$
so $t_{j-1}(d)= t_j(d)+t_j(d-d_j)$ for all $j$ and $d$.

Set
$$
s_i(d) = HF_{( \lambda_1, \dots, \lambda_i )}(d) = \dim (B/(\lambda_1, \dots, \lambda_{i}))_{d}
$$
Let
$$
K_i = \ker\left(B/(\lambda_1, \dots, \lambda_{i}) \overset{\lambda_i}{\longrightarrow} B/(\lambda_1, \dots, \lambda_{i-1}) \right),
$$
let $K_{i,d}$ denote the subspace of degree $d$ elements of $K_i$ and let $k_i(d) = \dim K_{i,d}$. Note that $\lambda_1, \dots, \lambda_m$ is $D$-semi-regular if and only if $k_i(d-d_i)=0$ for all $d <D$ and all $i =1, \dots, m$. 

We have an exact sequence
$$
0\to K_i \to B/(\lambda_1, \dots, \lambda_{i}) \overset{\lambda_i}{\longrightarrow} B/(\lambda_1, \dots, \lambda_{i-1}) \to B/(\lambda_1, \dots, \lambda_{i})\to 0
$$
which breaks up into exact sequences at degree $d$
$$
0\to K_{i,d-d_i} \to B/(\lambda_1, \dots, \lambda_{i})_{d-d_i} \overset{\lambda_i}{\longrightarrow} B/(\lambda_1, \dots, \lambda_{i-1})_d \to B/(\lambda_1, \dots, \lambda_{i})_d\to 0
$$
Taking the dimension of each term yields
$$
k_i(d-d_i) - s_i(d-d_i) + s_{i-1}(d)- s_i(d) =0
$$

We now prove the assertions in part (1) by induction on $m$ using the case $m=0$ (the ''empty sequence'') as the base case. In this situation the assertions follow from the fact that $HS_B(z) = (1+z)^n$. Now let $m >0$. The hypothesis of $D$-semi-regularity implies that  $s_{m-1}(d)= s_m(d) + s_m(d-d_m)$ for $d=0, \dots, D-1$. The induction hypothesis implies that $s_{m-1}(d)=t_{m-1}(d)$ for $d <D$. So
\begin{align*}
[HS_{(\lambda_1, \dots, \lambda_{m-1})}(z)]_D &= \sum_{d=0}^{D-1} s_{m-1}(d)z^d \\
	&= \sum_{d=0}^{D-1} s_{m}(d)z^d + \sum_{d=0}^{D-1} s_{m}(d-d_m)z^d \\
	&= (1+z^{d_m})\sum_{d=0}^{D-1} s_{m}(d)z^d\\
	&= (1+z^{d_m}) \left[HS_{(\lambda_1, \dots, \lambda_{m})}(z)\right]_D
\end{align*}
Using Lemma \ref{trunc} and induction on $m$ yields
\begin{align*}
[ HS_{(\lambda_1, \dots, \lambda_{m})}(z)]_D &=  \left[ \frac{1}{(1+z^{d_m})} [HS_{(\lambda_1, \dots, \lambda_{m-1})}(z)]_D \right]_D\\
&= \left[ \frac{1}{(1+z^{d_m})} \left[\frac{(1+z)^n}{\prod_{j=1}^{m-1}(1+z^{d_j})}\right]_D \right]_D\\
	&= \left[ \frac{(1+z)^n}{\prod_{j=1}^{m}(1+z^{d_j})}\right]_D
\end{align*}
which proves the first assertion. For the second part we assume, by induction, that $s_{m-1}(D) \geq t_{m-1}(D)$ and observe that by semi-regularity and the first part, $s_m(D-d_m)=t_m(D-d_m)$. Hence
\begin{align*} 
s_m(D) &= s_{m-1}(D) - s_m(D-d_m) +k_m(D-d_m)\\
	&\geq t_{m-1}(D) -t_m(D-d_m) = t_m(D)
\end{align*}

(2) Suppose that  $\lambda_1, \dots, \lambda_m$ is $D$-semi-regular but not $(D+1)$-semi-regular. Let $u$ be the smallest integer such that  $\lambda_1, \dots, \lambda_u$ is not semi-regular. Then $k_u(D-d_u)>0$, so
\begin{align*} 
s_u(D) &= s_{u-1}(D) - s_u(D-d_u) +k_u(D-d_u)\\
	&> t_{u-1}(D) -t_u(D-d_u) = t_u(D)
\end{align*}
Now suppose that $s_j(D) > t_j(D)$ for some  $u\leq j <m$. Then
\begin{align*}
s_{j+1}(D) &= s_{j}(D) - s_{j+1}(D-d_{j+1}) +k_{j+1}(D-d_{j+1})\\
	&> t_{j}(D) -t_{j+1}(D-d_{j+1}) = t_{j+1}(D)
\end{align*}
So by induction, $s_m(D)>t_m(D)$.

(3) Suppose now that the sequence  $\lambda_1, \dots, \lambda_m$  is semi-regular, and set $D =\ind(I)$. Then $[HS_I(z)]_D = [T_{{\bf d},n}(z)]_D$ by part (1) because the sequence is $D$-semi-regular. Because $D =\ind(I)$, 
$ HS_I(z) = [HS_I(z)]_D$. By  (1), $t_m(d) = s_m(d) > 0$ for $d<D$ and $t_m(D) \leq s_m(D) =0$, so  $\ind(T_{{\bf d},n}(z))=D$. Thus 
$$[T_{{\bf d},n}(z)]= [T_{{\bf d},n}(z)]_D = [HS_I(z)]_D =HS_I(z)$$
as required.

Conversely, suppose that $HS_I(z)  = [T_{{\bf d},n}(z)]$ and let $D= \ind(T_{{\bf d},n}(z))$. Then by definition, $D$ is the degree of regularity of the sequence $\lambda_1, \dots, \lambda_m$. If the sequence $\lambda_1, \dots, \lambda_m$ is not $D$-semi-regular, then there exists a $k<D$ such that it is $k$-semi-regular and not $(k+1)$-semi-regular. By part (2) we would then have that 
$$
s_m(k)>  t_m(k)
$$
That is, the $k$-th coefficient of $HS_I(z)$ is strictly greater than the $k$-th coefficient of $T_{{\bf d},n}(z)$, contradicting the fact that $HS_I(z)  = [T_{{\bf d},n}(z)]$. Thus the sequence is $D$-semi-regular and hence semi-regular.

(4) follows immediately from (3) because the Hilbert series of $B/I$ is independent of the order of the $\lambda_i$.
\end{proof}

It is natural to expect that information about the semi-regular sequences should give us information about arbitrary sequences. Since semi-regular sequences have as few relations as possible, we expect the ideal generated by a semi-regular sequence $(\nu_1, \dots, \nu_m)$ to grow at least as quickly as the ideal generated by an arbitrary sequence $(\lambda_1, \dots, \lambda_m)$. That is (if we use the notation $\sum a_iz^i \leq \sum b_iz^i \Leftrightarrow a_i \leq b_i$ for all $i$), 
$$
HS_{(\lambda_1, \dots, \lambda_m)}(z) \geq HS_{(\nu_1, \dots, \nu_m)}(z)
$$
Thus it is tempting to expect for any sequence $\lambda_1, \dots, \lambda_m$ that
$$
HS_{(\lambda_1, \dots, \lambda_m)}(z) \geq  \left[T_{{\bf d},n}(z)\right].
$$

The following example shows that this is not true.

\begin{example} Consider the element 
$$\lambda= x_1 x_2+x_3x_4+x_5x_6 +x_7x_8+x_9x_{10} + x_{11}x_{12}
$$
in $B^{(12)}$ and let $I =(\lambda)$. Then, using \cite[Theorem 2.1]{dhkst} we can calculate that
$$
HS_I(z) = 1 + 12z + 65z^2+ 208z^3 +430z^4 +584z^5+494 z^6 +208 z^7 +65 z^8 +12 z^9 +z^{10}
$$
while 
$$
\left[ \frac{(1+z)^{12}}{1+z^2} \right] = 1 + 12z + 65z^2+ 208z^3 +430z^4 +584z^5+494 z^6 +208 z^7 + z^8 +12 z^9 + 65 z^{10}
$$
Note also that in this case  $\ind((\lambda)) =  \ind(T_{(12),n}(z))$ but $\lambda$ is not semi-regular. Thus the condition $\ind(I) = \ind(T_{{\bf d},n}(z))$ is not equivalent to semi-regularity. 
\end{example}

It would be interesting to know whether $\ind((\lambda_1, \dots, \lambda_m)) \geq  \ind(T_{{\bf d},n}(z))$ for an arbitrary sequence $\lambda_1, \dots, \lambda_m$. All known evidence points to this result being true. However, the failure of the inequality $HS_I(z)\geq \left[T_{{\bf d},n}(z) \right]$ rules out the obvious way of proving it.
\numberwithin{thm}{section}

\section{Conjectures and Questions on Semi-Regularity}

It has long been conjectured that semi-regular sequences are in some sense ``generic''. However very little progress has been made towards proving this conjecture. In fact even the question of the existence of semi-regular quadratic sequences of length $n$ in $n$ of variables remains open. Let us begin by reviewing some of the conjectures made by Bardet et al.

\begin{Conjecture}\cite{b,bfsy} The proportion of semi-regular sequences tends to one as the number of variables tends to infinity.
\end{Conjecture}

We prove this conjecture in Section \ref{conj1} in the following precise sense. Recall that the semi-regularity of a sequence depends only on the set of elements, not on their order, so we actually consider the proportion of homogeneous subsets of $\Bn$  that are semi-regular. Let $h(n)$ be the number of subsets of $\Bn$ consisting of homogeneous elements of degree greater than or equal to one. Let $s(n)$ be the number of such subsets  that are semi-regular. We show in Theorem \ref{srgen} that 
$$
\lim_{n \to \infty} \frac{s(n)}{h(n)} = 1
$$
Unfortunately this result does not give us the kind of information that we are interested in. As the size of the set increases, so does the likelihood of it being semi-regular for trivial reasons (for instance any basis of the set of quadratic polynomials is trivially semi-regular). We show that the proportion of sequences that are trivially semi-regular tends to one.

A different formulation of the conjecture that most sequences are semi-regular is given in \cite[Conjecture 2]{bfs2}

\begin{Conjecture} \cite{bfs2} For any $(n,m,d_1, \dots, d_m)$ the proportion $\pi(n,m,d_1, \dots, d_m)$ of semi-regular sequences over $\bbF_2$ in the set $E(n,m,d_1, \dots, d_m)$ of algebraic systems of $m$ equations of degrees $d_1, \dots, d_m$ in $n$ variables tends to $1$ as $n$ tends to $\infty$.
\end{Conjecture}

We show that this conjecture is false. In fact, the opposite is true. For a fixed choice of $(m, d_1, \dots, d_m)$, Theorem \ref{nonexistance} states that
$$
\lim_{n\to \infty} \pi(n,m,d_1, \dots, d_m) =0.
$$

Neither of these conjectures accurately addresses the observed fact that ``most'' quadratic sequences of length $n$ in $n$ variables are semi-regular. More generally we make the following conjecture.

\begin{Conjecture} For any $1\leq d\leq n$ define $\pi(n,d)$ to be the proportion of sequences of degree $d$ and length $n$ in $n$ variables that are semi-regular. Then 
$$
\lim_{n\to \infty} \pi(n,d) =1.
$$
\end{Conjecture}

In fact we expect much more to be true. Define $\pi(n,m,d)$ to be the proportion of sequences of degree $d$ and length $m$ in $n$ variables that are semi-regular. The table below gives estimates of $\pi(n,m,2)$ for samples of 20 randomly chosen sequences of quadratic homogeneous polynomials.

\begin{table}[h] 
$$
\begin{array}{c|c|c|c|c|c|c|c|c|c|c|c|c|c|c|}
n\backslash m&2&3&4&5&6&7&8&9&10&11&12&13&14&15\\
\hline
3&1&.8&1&1&1&1&&&&&&&&\\
4&.35&1&.75&.75&.3&.65&.85&.9&1&1&1&1&1&1\\
5&0&.85&.95&1&.9&.85&.75&.6&.2&.65&.7&.9&.9&1\\
6&.85&.7&.65&.9&1&1&1&.95&.95&.95&.75&.8&.5&.25\\
7&0&.85&1&.1&1&1&1&1&1&1&1&.95&1&1\\
8&.7&.45&1&1&.95&.1&1&1&1&1&1&1&1&1\\
9&0&.95&.7&1&1&1&1&.8&.9&1&1&1&1&1\\
10&0&.85&1&.35&1&1&1&1&1&1&.25&1&1&1\\
11&0&.95&1&1&1&1&1&1&1&1&1&1&1&.4\\
12&0&0&1&1&1&1&.9&1&1&1&1&1&1&1\\
13&0&0&1&1&1&1&1&1&1&1&1&1&1&1\\
14&0&0&0&1&1&1&1&1&1&1&1&1&1&1\\
15&0&0&0&1&1&1&1&1&1&1&1&1&.45&1\\
\end{array}
$$
\caption{Proportion of Samples of 20 Sets of $m$ Homogeneous Quadratic Elements in $n$ variables that are Semi-Regular}
\end{table}

Theorem \ref{nonexistance} states that all columns of Table 1 eventually become zero. We conjecture that the non-zero entries of the rows tend to one as $n \to \infty$. One formulation of this is the following conjecture. 

\begin{Conjecture} \label{newconj} Fix $d>1$. Define $\pi(n,m, d)$ to be the proportion of sequences of degree $d$ and length $m$ in $n$ variables that are semi-regular. Then there exists an $0<\eta_d \leq 1/3$ such that  for all $\epsilon >0$, there exists an $N>0$  such that
$$
 \pi(n,m,d) >1-\epsilon \textrm{ for all } n > N \textrm{ and all } m>\eta N.
$$
\end{Conjecture}

While we believe these conjectures to be true, it should be noted that the existence question still remains largely open. 

\begin{question}For which pairs $(n, {\bf d}=(d_1, \dots, d_m))$  do there exist semi-regular sequences $\lambda_1, \dots, \lambda_m$ with $\deg \lambda_i = d_i$. Equivalently,  for which $n$ and ${\bf d} $ is 
$$
T_{{\bf d},n}(z)  =\frac{(1+z)^n}{\prod_{i=1}^m(1+z^{d_i})}
$$ 
the Hilbert series of an appropriate graded homomorphic image of the algebra $\grft$? 
\end{question}

At both ends of the degree spectrum, the existence question is trivial. Sequences of linear elements are semi-regular if and only if they are linearly independent. Likewise for sequences of degree $n-1$ (and $n$). Also for sufficiently large $m$ it is easy to find sequences that are trivially semi-regular; for instance a basis of the space of polynomials of degree $d$.

\section{The case $m=1$: semi-regularity of homogeneous polynomials}

In this section we give a complete answer to Question 1 in the case when the polynomial is linear or quadratic. We show that the proportion of semi-regular elements of degree $d$ in $B^{(n)}$ is zero when $n >3d$. Also, we give a complete description of the Hilbert Series and the index of a semi-regular element.

\begin{prop}
  Let $\lambda \in B^{(n)}_1$ then $\lambda$ is semi-regular and 
  $\ind(\lambda)=n$.
\end{prop}
\begin{proof}
  Without loss of generality we can assume that $$\lambda=x_1 \in B^{(n)}.$$In this case $B^{(n)}/(x_1) \cong B^{(n-1)}$ which has Hilbert series $(1+z)^{n-1}$. On the other hand $T_{(1),n}(z)=(1+z)^n/(1+z)=(1+z)^{n-1}$. So by Theorem \ref{bfsthm}, this element is semi-regular and $\ind (\lambda)= n$.
\end{proof}

\begin{thm}\label{trivialsem}
Let $\lambda \in B^{(n)}$ be homogeneous of degree $d$. If $d=n$ or $d=n-1$ then $\lambda$ is semi-regular . 
\end{thm}

\begin{proof}
Let $\lambda \in B^{(n)}$ be homogeneous of degree $d$, with $n-1\leq d \leq n$. Note that $\ind (\lambda)=n$. The semi-regularity of $\lambda$ follows trivially from the definition.
\end{proof}

\begin{lem}
  Let $\lambda \in B^{(n)}$ be a monomial then $\Ann \lambda = (\var(\lambda))$ where $\var(\lambda)$ is the set of variables occurring in $\lambda$.
\end{lem}
\begin{proof}
  The inclusion $(\var(\lambda)) \subset \Ann \lambda$ is clear. Let 
  $\nu \in \Ann \lambda$ and write $\nu=\nu_1+\cdots +\nu_r$, where the $\nu_i$ are distinct monomials. Since 
  $\lambda$ is a monomial for $i \neq j$ it follows that 
  $\nu _i \lambda \neq \nu_j \lambda$ unless $\nu _i \lambda = \nu_j \lambda=0$. Thus $\lambda \nu_j = 0$ and so
  $\lambda$ and $\nu_j$ must share some $x_i$. Hence 
  $\nu_j \in (\var(\lambda))$. 
\end{proof}

\begin{defn}
Let $\lambda \in B^{(n)}$ be a homogeneous polynomial of degree $d$. $\lambda$ can be written in the form
$$\lambda=\sum_{m\in \mu_d}\epsilon_mm,$$where $\mu_d$ is the set of monomials of degree $d$, and $\epsilon_m$ is either $1$ or $0$. The \textit{support} of $\lambda$ is defined as
$$\Supp(\lambda)=\{m \mid \epsilon_m=1\}.$$
\end{defn}

\begin{prop} \label{indlam}
 Let $\lambda \in B^{(n)}$ be homogeneous of degree $d$. Then $\ind(\lambda)> n-d$. 
\end{prop}
\begin{proof}
  By renumbering we may assume $x_1\cdots x_d \in \Supp(\lambda)$. We 
  demonstrate $x_{d+1}\cdots x_n \notin (\lambda)$. Suppose for 
  the sake of contradiction that we have $\nu \in B^{(n)}_{n-2d}$ such that
  $\nu \lambda = x_{d+1}\cdots x_n$. Writing $\nu$ and $\lambda$ as
  polynomials in $x_1$, i.e. $\lambda=x_1\lambda_1+\lambda_0$ and $\nu=x_1\nu_1+\nu_0$ then  
  \[\lambda_0\nu_0+x_1(\lambda_1\nu_0+\lambda_0\nu_1)=x_{d+1}\cdots x_n.\]
  So $\lambda_1 \nu_0+\lambda_0\nu_1 \in \Ann x_1$, but $x_1 \notin
  \var(\lambda_1 \nu_0+\lambda_0\nu_1)$ therefore $\lambda_1 \nu_0=
  \lambda_0\nu_1$. In particular 
  \[x_{d+1} \cdots x_n \lambda_1 = \lambda_0\nu_0 \lambda_1=\lambda_0^2\nu_1=0,\]
  so $\lambda_1 \in \Ann x_{d+1}\cdots x_n=(x_{d+1},\cdots,x_n)$ 
  but $x_2\cdots x_d \in \Supp \lambda_1$, which is impossible.
\end{proof}
Now, we will use a result that appears in \cite{hps} about the $\textit{first fall degree}$ of a homogeneous polynomial $\lambda \in B^{(n)}$. Basically, the first fall degree of $\lambda$ is the first degree at which non-trivial relations occur; trivial relation such as $g\lambda=0$ where $g \in (\lambda)$. In other words the first fall degree of $\lambda$  is the first $k$ such that there exists $g$ in $B^{(n)}$ with the property  that $\deg g + \deg \lambda=k$, $g\lambda=0$ and $g \not\in (\lambda)$. In \cite{hps} the authors give a more detailed and general definition for the first fall degree.

\begin{defn}
Let $\lambda$ be a homogeneous element of $B^{(n)}$. The \textit{rank} of $\lambda$ is the smallest integer $s$ such that there exist $\mu_1,\dots,\mu_s \in B^{(n)}_1$ with  $\lambda\in \bbF_2[\mu_1,\dots,\mu_s]$. That is, $s$ is the smallest number of linear elements required to generate $\lambda$.
\end{defn}

\begin{thm} \label{ffdelement}
  Let $\lambda$ be an element of degree $d>1$ and rank $s$. Then 
  $\ffd (\lambda) \leq (s+d+2)/2$. 
\end{thm}
\begin{proof}
  See Theorem 4.9 in \cite{hps}.
\end{proof}

This enables us to give a result on the non-existence of semi-regular elements of degree $d\geq 2$ when $n >3d$.

\begin{thm}\label{indffd2}
Let $\lambda_1,\dots,\lambda_m$ be a sequence of homogeneous polynomials of degrees $d_1,\dots,d_m$ in $B^{(n)}$. If $\ind (\lambda_1,\dots,\lambda_m)>\ffd (\lambda_i)$ for some $1\leq i \leq m$ then $\lambda_1,\dots,\lambda_m$ is not semi-regular.
\end{thm}

\begin{proof}
Suppose $\lambda_1,\dots,\lambda_m$ is a sequence of homogeneous polynomials such that $\ind (\lambda_1,\dots,\lambda_m)>\ffd (\lambda_i)$ for some $1\leq i \leq m$. If $\lambda_1,\dots,\lambda_m$ is a semi-regular sequence then by Theorem \ref{bfsthm} any reordering of this sequence is also a semi-regular sequence. Thus, without loss of generality we can assume that $\ind (\lambda_1,\dots,\lambda_m)>\ffd (\lambda_1)$. By definition of first fall degree there exists $g$ such that $g\lambda_1=0$, $\deg g + \deg \lambda_1=\ffd(\lambda_1)<\ind (\lambda_1,\dots,\lambda_m)$ and $g \not\in (\lambda_1)$. But it is not possible if $\lambda_1,\dots,\lambda_m$ is semi-regular. Therefore $\lambda_1,\dots,\lambda_m$ cannot be semi-regular.
\end{proof}

\begin{cor}\label{indffd}
Let $\lambda$ be homogeneous. If $\ind (\lambda)>\ffd (\lambda)$ then $\lambda$ is not semi-regular.
\end{cor}

\begin{thm}\label{thmbound}
  There are no semi-regular elements of degree $d\geq 2$ for $n > 3d$.
\end{thm}
\begin{proof}  Let $\lambda$ be a homogeneous element with $\deg \lambda =d>1$ and suppose that $n>3d$. Then $(n+d)/2<n-d$. Since the rank $s$ of $\lambda$ is less than or equal to $n$, we have by Theorem \ref{ffdelement} and Proposition \ref{indlam} that
$$
\ffd (\lambda) \leq \frac{s+d+2}{2}\leq \frac{n+d+2}{2} < n-d+1 \leq \ind(\lambda).
$$
Since the first fall degree of $\lambda$ is less than its index it cannot be semi-regular.
\end{proof}

Next theorem gives a complete description of the proportion of quadratic  semi-regular elements.
\begin{thm} \label{m1d2} There are no semi-regular elements of degree 2 for $n \geq 7$. That is, $\pi(n,1,2) = 0$ for $n\geq7$. For $2\leq n\leq 6$ the value of $\pi(n,1,2)$ is given by the table
$$\begin{array}{c|c|c|c|c|c}
n & 2&3&4&5&6\\
\hline
\pi(n,1,2) &\ 1\ &\ 1\ &0.44&0.85&0.42\\
\end{array}
$$
\end{thm}

\begin{proof}First part is just a consequence of Theorem \ref{thmbound}. 

For the cases $n=2,\dots,6$, we can compute the Hilbert series of a polynomial of a given rank using the specific case of $x_1x_2, x_1x_2+x_3x_4$ and $x_1x_2+x_3x_4+x_5x_6$. The cases $n=2$ and $3$ follow from Theorem \ref{trivialsem}. When $n=4$ the Hilbert series of a rank two element is $1+4z+5z^2+2z^3$ and that of a rank 4 element is $1+4z+5z^2$ which is $[T_{(2),4}(z)]$. Thus the rank four elements are semi-regular and the rank two elements are not. There are 28 quadratic homogeneous elements of rank four and 35 elements of rank two. Thus the proportion of semi-regular elements is $28/63=0.44$. In the case $n=5$, the Hilbert series of a rank two element is $1+5z+9z^2+7z^3+2z^4$ and that of a rank 4 element is $1+5z+9z^2+5z^3+z^4 =[T_{(2),5}(z)]$. The number of elements of ranks two and four  is respectively $155$ and $868$ yielding a proportion of semi-regular elements of $0.85$. When $n=6$ the Hilbert series of a rank two element is $1+6z+14z^2+16z^3+9z^4+2z^5$, that of a rank 4 element is $1+6z+14z^2+14z^3+5z^4$ and that of a rank six element is $1+6z+14z^2+14z^3+z^4 = [T_{(2),6}(z)]$. The number of elements of ranks two, four and six is respectively $651$, $18,228$ and $13,888$ yielding a proportion of semi-regular elements of $0.43$.
\end{proof}

In her thesis \cite{b}, Bardet asserts that the element $\sum_{1\leq i<j\leq n} x_ix_j$ is semi-regular for all $n$ over $\bbF_2$. Theorem \ref{m1d2} implies that this element is in fact not semi-regular for any $n\geq 7$.

Theorem \ref{thmbound} tells us that there are not semi-regular elements of degree $d\geq2$ for $n>3d$. We consider the case of a single homogeneous element of arbitrary degree $d\leq n/3$. The table below gives some experimental data for samples of 20 homogeneous polynomials of degree $d$ in $n$ variables.
\begin{table}[h]
$$
\begin{array}{c|c|c|c|c|c|c|c|c|c|c|c|c|c|}
n\backslash d&2&3&4&5&6&7&8&9&10\\
\hline
4&.5&1&1&&&&&&\\
5&.9&0&1&1&&&&&\\
6&.45&1&.95&1&1&&&&\\
7&&0&1&0&1&1&&& \\
8&&1&.25&1&.25&1&1&&\\
9&&0&1&.65&1&0&1&1&\\
10&&&.5&1&0&1&.5&1&1\\
\end{array}
$$
\caption{Proportion of Samples of 20 Homogeneous Elements of Degree $d$  in $n$ variables that are Semi-Regular}
\end{table}

Note that the ones on the upper two diagonals reflect that fact that all elements of degree $n-1$ or $n$ are semi-regular, whereas the ones on the other diagonals reflect only a high probability of semi-regularity since a monomial of degree less than $n-1$ is never semi-regular. In the following section we show that if $d=2^t$ and $n=3d$ then $\sigma_{n,d}$ is semi-regular, thus establishing that the bound $d\leq n/3$ is sharp for infinitely many $n$.

Now, we will give a complete description of the truncated series
$$\left[ \frac{(1+z)^n}{1+z^d}\right ]$$
when $n\leq 3d$. First, note that 
\begin{align*}
\frac{(1+z)^n}{1+z^d}&=(1+z)^n(1-z^d+z^{2d}+\cdots+(-1)^jz^{jd}+\cdots).\\
\end{align*}
Therefore,
\begin{align}\label{eq:eq1}
\frac{(1+z)^n}{1+z^d}&=\sum_{k=0}^\infty\gamma(n,k,d)z^k
\end{align}
where 
$$\gamma (n,k,d)=\sum_{j=0}^{\lfloor k/d \rfloor}(-1)^j\binom{n}{k-jd}.$$

\begin{lem}\label{lemma bi}
Let $n$, and $d$ be two natural numbers.
\begin{enumerate}[(a)]
\item
If $k$ is a non-negative integer number such that $k <\lceil (n+d)/2 \rceil$ then
$$\binom{n}{k}-\binom{n}{k-d}>0$$
\item
If $n+d$ is odd and $k=\lceil (n+d)/2\rceil=(n+d+1)/2$ then
$$\binom{n}{k}-\binom{n}{k-d}+\binom{n}{0}\leq 0$$
\end{enumerate}
\end{lem}

\begin{proof}
Suppose $k$ is a non-negative integer number such that $k < \lceil(n+d)/2 \rceil$. Note that $\lceil (n+d)/2 \rceil=(n+d)/2$ or $\lceil (n+d)/2 \rceil=(n+d+1)/2$. In any case since $k$ is integer we have that $k<(n+d)/2$. Also note that 
$$\binom{n}{k}-\binom{n}{k-d}>0$$
if $k\leq n/2$. Now suppose that $n/2 \leq k <(n+d)/2$. Then $k-d<n-k\leq n/2$ so
$$\binom{n}{k}-\binom{n}{k-d}=\binom{n}{n-k}-\binom{n}{k-d}>0$$
proving (a).
Now let us suppose that $\lceil (n+d)/2\rceil=(n+d+1)/2=k$. Let us prove that
$$\binom{n}{k}-\binom{n}{k-d}+\binom{n}{0}\leq 0.$$
Since $(n+d+1)/2=k$ then $(n+d)/2<k$, thus $n-k<k-d$. Also, since $k= (n+d+1)/2 \leq (n+2d)/2$ then $k-d\leq n/2$. Therefore,
$$\binom{n}{k}-\binom{n}{k-d}=\binom{n}{n-k}-\binom{n}{k-d}<0$$
so
$$\binom{n}{k}-\binom{n}{k-d}+\binom{n}{0}\leq 0.$$
\end{proof}

\begin{thm} \label{sertrun1}
Let $n$, and $d$ be two natural numbers.
If $n<3d$ then 
\begin{equation*}
\left[\frac{(1+z)^n}{1+z^d}\right]=\sum_{k=0}^{D-1}\left[\binom{n}{k}-\binom{n}{k-d}\right]z^k
\end{equation*}
where $D=\lceil (n+d)/2\rceil$.
\end{thm}

\begin{proof}
Since
\begin{equation*}
\frac{(1+z)^n}{1+z^d}=\sum_{k=0}^\infty\gamma(n,k,d)z^k
\end{equation*}
we want to show that $\gamma(n,k,d)>0$ if $k< \lceil (n+d)/2\rceil$   and that  $\gamma(n,k,d)\leq 0$ when $k=\lceil (n+d)/2\rceil$.
Since $n<3d$ then $(n+d)/2<2d$, thus $\lceil (n+d)/2\rceil\leq 2d$. Let $k$ be a non-negative integer such that $k<\lceil (n+d)/2\rceil\leq 2d$. By Lemma \ref{lemma bi} we have that
$$\gamma(n,k,d)=\binom{n}{k}-\binom{n}{k-d} >0.$$
Suppose now that $k=\lceil (n+d)/2\rceil$. If $n+d$ is even, then $k=(n+d)/2< (3d+d)/2 =2d$ and $n-k=k-d$. So
$$
\gamma(n,k,d)=\binom{n}{k}-\binom{n}{k-d}=\binom{n}{k}-\binom{n}{n-k}=0.
$$
If $n+d$ is odd then $k=(n+d+1)/2$ and  $k \leq 2d$. Hence
$$
\gamma(n,k,d)\leq \binom{n}{k}-\binom{n}{k-d}+\binom{n}{0} \leq 0
$$
 by Lemma \ref{lemma bi}. The result is proved.
\end{proof}

\begin{lem}
Let $d\geq 2$. Then
$$\frac{(3d)!}{(2d+1)!(d+1)!}\geq 1$$
\end{lem}
\begin{proof}
For $d=2$ we have that
$$\frac{(3d)!}{(2d+1)!(d+1)!}=\frac{6!}{5!3!}=1$$
Suppose the result is true for $d$ let us prove it for $d+1$. By induction we have that
\begin{align*}
\frac{(3(d+1))!}{(2(d+1)+1)!(d+2)!}&=\frac{(3d)!}{(2d+1)!(d+1)!}\frac{(3d+1)(3d+2)(3d+3)}{(2d+2)(2d+3)(d+2)}\\
&\geq \frac{(3d+1)(3d+2)(3d+3)}{(2d+2)(2d+3)(d+2)}
\end{align*}
To show that 
$$\frac{(3d+1)(3d+2)(3d+3)}{(2d+2)(2d+3)(d+2)}\geq 1$$
is equivalent to show that 
$$23d^3+36d^2+7d-6\geq 0.$$
The last inequality is true since for $d\geq 2$ we have that $23d^3+36d^2+7d\geq 6$.
\end{proof}

\begin{lem}\label{lem bin}
Let $d\geq 2$. Then
$$\binom{3d}{2d+1}-\binom{3d}{d+1}+\binom{3d}{1}<0$$
\end{lem}

\begin{proof}
Since $d\geq2$, by above lemma we have
\begin{align*}
\binom{3d}{2d+1}-\binom{3d}{d+1}+\binom{3d}{1}&=\frac{(3d)!}{(2d+1)!(d-1)!}-\frac{(3d)!}{(d+1)!(2d-1)!}+3d\\
&=\frac{(3d)!(d(d+1))}{(2d+1)!(d+1)!}-\frac{(3d)!(2d(2d+1))}{(d+1)!(2d+1)!}+3d\\
&=\frac{(3d)!}{(2d+1)!(d+1)!}(-3d^2-d)+3d\\
&\leq -3d^2-d+3d\\
&= -3d^2+2d\\
&=d(-3d+2)<0
\end{align*}
\end{proof}

\begin{thm} \label{sertrun2}
Let $n$, and $d$ be two natural numbers. 
If $n=3d$ and $d\geq 2$ then 
\begin{equation*}
\left[\frac{(1+z)^n}{1+z^d}\right]=\sum_{k=0}^{2d-1}\left[\binom{n}{k}-\binom{n}{k-d}\right]z^k+z^{2d}
\end{equation*}
\end{thm}

\begin{proof}
We know that 
\begin{equation*}
\frac{(1+z)^n}{1+z^d}=\sum_{k=0}^\infty\gamma(n,k,d)z^k.
\end{equation*}
Suppose $n=3d$. Then for $k<2d$, 
$$
\gamma(n,k,d)= \binom{3d}{k}-\binom{3d}{k-d} > 0
$$
by Lemma \ref{lemma bi}. Also
$$
\gamma(n,2d,d) = \binom{3d}{2d}-\binom{3d}{d} +  \binom{3d}{0} = 1
$$
and by Lemma \ref{lem bin}
$$
c_{2d+1} = \binom{3d}{2d+1}-\binom{3d}{d+1} +  \binom{3d}{1} <0.
$$
It proves the result.
\end{proof}

\begin{thm}\label{corind}
 Suppose that $\lambda$ is a semi-regular homogeneous element of degree $d>1$. Then $n \leq 3d$ and
$$
\ind(\lambda)=\begin{cases} \ceil{(n+d)/2} & \mathrm{if} \ n < 3d \\
				(n+d+2)/2=2d+1 & \mathrm{if} \ n = 3d 
\end{cases}
$$ 
\end{thm}

\begin{proof}
If $\lambda$ is semi-regular, then by Theorem \ref{bfsthm}
$$HS_{(\lambda)}(z)=\left[\frac{(1+z)^n}{1+z^d}\right]$$
thus $\ind (\lambda)=\ind (1+z)^n/(1+z^d)$ so the result follows from Theorem \ref{sertrun1} and Theorem \ref{sertrun2}.
\end{proof}

To finish this section, we present a theorem that will be used in the following section where we prove some results about semi-regularity of the elementary symmetric polynomials.
\begin{lem}\label{lem1}
Let $\lambda \in B^{(n)}$ be a homogeneous element of degree $d$. Suppose that for $k < n$ the map 
$$B^{(n)}_{k}\xrightarrow{\lambda}B^{(n)}_{k+d}$$ 
multiplication by $\lambda$, is injective. Then the map
$$B^{(n)}_{k-1}\xrightarrow{\lambda}B^{(n)}_{k-1+d}$$ is injective.
\end{lem}

\begin{proof}
Suppose that for $k<n$ the map 
$$B^{(n)}_{k}\xrightarrow{\lambda}B^{(n)}_{k+d}$$ 
is injective. Suppose that there exists $\alpha\neq 0$ such that $\alpha\in B^{(n)}_{k-1}$ such that $\alpha \lambda=0$. Since $k<n$ we have that there exits $x_i$ such that $x_i \nmid \alpha$. So $x_i\alpha\neq 0$ and $x_i\alpha \in B^{(n)}_{k}$ satisfies that $x_i\alpha \lambda=0$, which is a contradiction.
\end{proof}

\begin{thm}\label{th2}
Let $\lambda \in B^{(n)}$ be a homogeneous element of degree $d$. Suppose $d=\deg(\lambda)\geq n/3$.
\begin{enumerate}[(a)]
\item
If $d>n/3$ then $\lambda$ is semi-regular if and only if for $D=\lceil (n+d)/2\rceil$ the map 
$$B^{(n)}_{D-1-d}\xrightarrow{\lambda}B^{(n)}_{D-1}$$ 
is injective, and the map
$$B^{(n)}_{D-d}\xrightarrow{\lambda}B^{(n)}_D$$
is surjective.
\item
If $d=n/3$ then $\lambda$ is semi-regular if and only if the maps 
$$B^{(n)}_{d-1}\xrightarrow{\lambda}B^{(n)}_{2d-1}$$ 
and
$$\frac{B^{(n)}_d}{\lambda B^{(n)}_0} \xrightarrow{\lambda}B^{(n)}_{2d}$$ 
are injective and the map
$$B^{(n)}_{d+1}\xrightarrow{\lambda}B^{(n)}_{2d+1}$$
is surjective.
\end{enumerate}
\end{thm}

\begin{proof}
Suppose $n=3d$. Suppose $\lambda$ is semi-regular. By Theorem \ref{bfsthm} and Theorem \ref{sertrun2} we have
\begin{align*}
{HS}_{(\lambda)}(z)&=\sum_{k=0}^{\infty}\dim \left(B^{(n)}_{k}/\lambda B^{(n)}_{k-d}\right)z^k \\
&=\left[\frac{(1+z)^n}{1+z^d}\right]=\sum_{k=0}^{2d-1}\left[\binom{n}{k}-\binom{n}{k-d}\right]z^k+z^{2d}.
\end{align*}
Thus 
\begin{align*}
\dim (B^{(n)}_{2d-1})-\dim(\lambda B^{(n)}_{d-1})&=\dim \left(B^{(n)}_{2d-1}/\lambda B^{(n)}_{d-1}\right)\\
&=\binom{n}{2d-1}-\binom{n}{d-1}\\
&=\dim (B^{(n)}_{2d-1})-\dim(B^{(n)}_{d-1}).
\end{align*}
So $\dim (\lambda B^{(n)}_{d-1})=\dim (B^{(n)}_{d-1})$, then we have that the map
$$B^{(n)}_{d-1}\xrightarrow{\lambda}B^{(n)}_{2d-1}$$ 
is injective. Now, since $n=3d$ note that $\dim B^{(n)}_{2d}=\dim B^{(n)}_d$ and
\begin{align*}
\dim (B^{(n)}_{2d})-\dim(\lambda B^{(n)}_{d})&=\dim (B^{(n)}_{2d}/\lambda B^{(n)}_{d})\\
&=1\\
&=\dim (B^{(n)}_{2d})-(\dim(B^{(n)}_{d})-1).
\end{align*}
So $\dim (\lambda B^{(n)}_{d})=\dim (B^{(n)}_{d})-1$, then the kernel of the map
$$B^{(n)}_{d}\xrightarrow{\lambda}B^{(n)}_{2d}$$
has dimension $1$, however $\lambda \in B^{(n)}_d$ is in this kernel. Therefore the kernel should be $\lambda B^{(n)}_0$. So, the map
$$\frac{B^{(n)}_d}{\lambda B^{(n)}_0} \xrightarrow{\lambda}B^{(n)}_{2d}$$ 
is injective. Finally, $\dim (B^{(n)}_{2d+1}/\lambda B^{(n)}_{d+1})=0$ then we have that 
$$B^{(n)}_{d+1}\xrightarrow{\lambda}B^{(n)}_{2d+1}$$
is surjective. Conversely, suppose that the maps 
$$B^{(n)}_{d-1}\xrightarrow{\lambda}B^{(n)}_{2d-1}$$ 
and
$$\frac{B^{(n)}_d}{\lambda B^{(n)}_0} \xrightarrow{\lambda}B^{(n)}_{2d}$$ 
are injective and the map
$$B^{(n)}_{d+1}\xrightarrow{\lambda}B^{(n)}_{2d+1}$$
is surjective. Since the map 
$$B^{(n)}_{d-1}\xrightarrow{\lambda}B^{(n)}_{2d-1}$$ 
is injective, then by Lemma \ref{lem1} we have that 
$$B^{(n)}_{k-d}\xrightarrow{\lambda}B^{(n)}_{k}$$ 
is injective, for all $k < 2d$. Thus, by Lemma \ref{lemma bi} we have that for all $k<2d$
\begin{align*}
\dim \left(B^{(n)}_{k}/\lambda B^{(n)}_{k-d}\right)&=\dim (B^{(n)}_{k})-\dim(\lambda B^{(n)}_{k-d})\\
&=\dim (B^{(n)}_{k})-\dim(B^{(n)}_{k-d})\\
&=\binom{n}{k}-\binom{n}{k-d}>0.
\end{align*}
Since the map
$$\frac{B^{(n)}_d}{\lambda B^{(n)}_0} \xrightarrow{\lambda}B^{(n)}_{2d}$$ 
is injective and $\dim (B^{(n)}_{2d})=\dim(B^{(n)}_d)$ then $\dim (B^{(n)}_{2d}/\lambda B^{(n)}_d)=1$. Finally, since the map
$$B^{(n)}_{d+1}\xrightarrow{\lambda}B^{(n)}_{2d+1}$$
is surjective  then $\dim (B^{(n)}_{2d+1}/(\lambda B^{(n)}_{d+1}))=0$. Putting together this information we have that
 \begin{align*}
{HS}_{(\lambda)}(z)&=\sum_{k=0}^{\infty}\dim \left(B^{(n)}_{k}/\lambda B^{(n)}_{k-d}\right)z^k \\
&=\sum_{k=0}^{2d-1}\left[\binom{n}{k}-\binom{n}{k-d}\right]z^k+z^{2d}.
\end{align*}
By Theorem \ref{sertrun2} 
$$\left[\frac{(1+z)^n}{1+z^d}\right]= \sum_{k=0}^{2d-1}\left[\binom{n}{k}-\binom{n}{k-d}\right]z^k+z^{2d}.$$
Thus, 
$${HS}_{(\lambda)}(z)=\left[\frac{(1+z)^n}{1+z^d}\right].$$
So $\lambda$ is semi-regular. It proves (b). The proof of (a) is similar.
\end{proof}

\section{Semi-regularity of elementary symmetric polynomials}

Consider the ring of polynomials over $\mathbb{F}_2$ in $n$ variables, $\mathbb{F}_2[X_1,...,X_n]$. In this ring we have the elementary symmetric polynomial of degree $d$ which is defined as 
\begin{equation*}
\sigma_d(X_{1},...,X_{n})=\sum_{1\leq i_1<\cdots<i_d\leq n}X_{i_1}\cdots X_{i_d}.
\end{equation*}
We can consider the elementary symmetric polynomial of degree $d$ in $B^{(n)}$,\\* $\sigma_d(x_{1},...,x_{n})$, as the image of the symmetric polynomial $\sigma_d(X_{1},...,X_{n})$ under the evaluation map
\begin{align*}
\mathbb{F}_2[X_1,...,X_n]&\rightarrow B^{(n)}\\
X_i &\mapsto x_i
\end{align*} 
To simplify we will denote $\sigma_d(x_1,...,x_n)$ by $\sigma_{d,n}$. 

In this section we give a complete description of the semi-regularity of the elementary symmetric polynomials $\sigma_{d,n}$. First, recall the following well-known identity. 

\begin{lem}
Let $n$, $d$, and $k$ be natural numbers such that $1\leq k \leq n$ then 
\begin{align}\label{re1}
\sigma_d(x_1,...,x_n)=\sum_{i=0}^{d}\sigma_{d-i}(x_1,...,x_k)\sigma_i(x_{k+1},...,x_n)
\end{align}
\end{lem}

\begin{proof}
Note that 
\begin{align*}
\sum_{j=1}^n \sigma_j(x_1,...,x_n) t^j &=\prod_{i=1}^n (1+tx_i) \\
	&= \prod_{i=1}^k (1+tx_i) \prod_{j=1}^{n-k} (1+ tx_j)\\
	&= \sum_{i=1}^k \sigma_i(x_1,...,x_k) t^i \sum_{j=1}^{n-k} \sigma_j(x_{k+1},...,x_n) t^j 
\end{align*}
\end{proof}

\begin{lem}\label{lemmul}
$$\sigma_{a,n}\sigma_{b,n} = \overline{\binom{a+b}{a}} \sigma_{a+b,n}
$$
where $\overline{k}$ denotes the image of $k$ in $\bbF_2$.
\end{lem}

\begin{proof} Let $M$ be a monomial in $x_1, \dots , x_n$ of degree $a+b$. Then $M$ will occur once in $\sigma_{a,n}\sigma_{b,n}$ for each occurrence of a sub-monomial of $M$ of degree $a$ in $\sigma_{a,n}$. There are precisely $\binom{a+b}{a}$ such sub-monomials.
\end{proof}

\begin{cor}
$$
\sigma_{1,n}\sigma_{k,n}=\begin{cases} 0 & \textrm{if } k \textrm { is odd}\\
				 \sigma_{1+k,n} & \textrm{if } k \textrm { is even}\\
\end{cases}
$$
\end{cor}

\begin{thm}\label{gamsem1} 
If $d$ is odd, then $\sigma_{d,n}$ is semi-regular if and only if $d = n-1, n$.
\end{thm}

\begin{proof} Suppose that $\sigma_{d,n}$ is semi-regular. By above corollary, the first fall degree of $\sigma_{d,n}$ is less than or equal to $d+1$. On the other hand, we know that 
$$
\ind(\sigma_{d,n}) \geq \left\lceil \frac{n+d}{2}\right\rceil
.$$
Since $\sigma_{d,n}$ is semi-regular by Corollary \ref{indffd} we have that $\ind (\sigma_{d,n}) \leq \ffd (\sigma_{d,n})$, thus $\lceil(n+d)/2\rceil \leq d+1$. This implies that $n+d \leq 2d+2$ from which we obtain $d \geq n-2$. If $d = n-2$, then $\ind (\sigma_{d,n})=\left\lceil (n+d)/2\right\rceil = n-1$. Thus, $\dim(B^{(n)}_{n-1}/\sigma_{d,n}B^{(n)}_1)=0$ then multiplication by $\sigma_{d,n}$ would be an surjective map from $B^{(n)}_1$ to $B^{(n)}_{n-1}$; however, $\dim B^{(n)}_{n-1}=\dim B^{(n)}_1=n$, then multiplication by $\sigma_{d,n}$ would be an injective map from $B^{(n)}_1$ to $B^{(n)}_{n-1}$; but $\sigma_{1,n} \sigma_{d,n} =0$ by above corollary so this does not happen. Hence we must have that $d=n-1$ or $n$. Conversely, any element $\lambda \in B^{(n)}_n \cup B^{(n)}_{n-1}$ is semi-regular as we show in Theorem \ref{trivialsem}.
\end{proof}

\begin{lem}\label{leminj}
Suppose that for $k<n$ the map 
$$B^{(n)}_{k}\xrightarrow{\sigma_{d,n}}B^{(n)}_{k+d}$$ 
is injective. Then the map
$$B^{(n+1)}_{k}\xrightarrow{\sigma_{d,n+1}}B^{(n+1)}_{k+d}$$ is injective.
\end{lem}

\begin{proof}
Suppose that the map 
$$B^{(n)}_{k}\xrightarrow{\sigma_{d,n}}B^{(n)}_{k+d}$$ 
is injective. Suppose that there exists $\alpha\in B^{(n+1)}_{k}$ such that $\alpha \sigma_{d,n+1}=0$. By equation $(\ref{re1})$ we have that $\sigma_{d,n+1}=\sigma_{d,n}+x_{n+1}\sigma_{d-1,n}$, also notice that $\alpha=\alpha'+x_{n+1}\beta$ with $\alpha' \in B^{(n)}_{k}$ and $\beta \in B^{(n)}_{k-1}$. Thus $0=\alpha \sigma_{d,n+1}=\alpha' \sigma_{d,n}+x_{n+1}(\alpha'\sigma_{d-1,n}+\beta \sigma_{d,n})$. Since $\alpha' \sigma_{d,n}\in B^{(n)}$ we have that $\alpha' \sigma_{d,n}=0$ so by injectivity $\alpha'=0$. Thus, $0=x_{n+1}(\beta \sigma_{d,n})$, but $\beta \sigma_{d,n}\in B^{(n)}$, so $\beta \sigma_{d,n}=0$. Therefore by Lemma $\ref{lem1}$ we have that $\beta=0$. So $\alpha=0$.
\end{proof}

\begin{thm}\label{th1}
Let $n$ and $k$ be non-negative integers. Then \[
\overline{\binom{n}{k}}=
 \begin{cases}
     \hfill 0 \hfill & \text{if $n$ is even and $k$ is odd}\\
     \hfill \overline{\binom{\lfloor n/2\rfloor} {\lfloor k/2 \rfloor}} \hfill & \text{otherwise}

\end{cases}
\]
where $\overline{a}$ denotes the image of $a$ in $\mathbb{F}_2$.
\end{thm}

\begin{proof}
See Theorem 4.1.10 in \cite{jlg}.
\end{proof}

\begin{lem}\label{lemco}
Let $d=2^nl$. Then for all $1\leq k \leq 2^n-1$ we have that
\[
\overline{\binom{2^nl+j}{k}}=
 \begin{cases}
     \hfill 1 \hfill & \text{if \hfill$j=k$}\\
     \hfill 0 \hfill & \text{if \hfill$ 0\leq j \leq k-1$}
\end{cases}
\]
where $\overline{a}$ denotes the image of $a$ in $\mathbb{F}_2$.
\end{lem}

\begin{proof}
If $n=1$, then $k=1$ and clearly we have that 
$\overline{\binom{2l+1}{1}}=1$ and $\overline{\binom{2l}{1}}=0$. Suppose by induction that the result is true for $n$, let us prove it for $n+1$. For $k=1$ we have that $\overline{\binom{2^{n+1}l+1}{1}}=1$ and $\overline{\binom{2^{n+1}l}{1}}=0$. Suppose $2 \leq k \leq 2^{n+1}-1$ and $0\leq j \leq k-1$, then $1\leq \lfloor k/2 \rfloor \leq 2^n-1$, and $0 \leq \lfloor j/2 \rfloor \leq \lfloor k/2 \rfloor -1$. Suppose first that $k$ is even. Then by Theorem $\ref{th1}$ and by induction we have that
\[
\overline{\binom{2^{n+1}l+j}{k}}=\overline{\binom{2^{n}l+\lfloor j/2 \rfloor}{\lfloor k/2 \rfloor}}=
 \begin{cases}
     \hfill 1 \hfill & \text{if $j=k$}\\
     \hfill 0 \hfill & \text{if $ 0\leq j \leq k-1$}
\end{cases}
\]
Suppose now that $k$ is odd. Note that $2^{n+1}l+k$ is odd then by Theorem $\ref{th1}$ and by induction we have that
$$\overline{\binom{2^{n+1}l+k}{k}}=\overline{\binom{2^{n}l+\lfloor k/2 \rfloor}{\lfloor k/2 \rfloor}}=1.$$
Now if $0\leq j \leq k-1$ is even then by Theorem $\ref{th1}$ we have that 
$$\overline{\binom{2^{n+1}l+j}{k}}=0.$$
If $0\leq j \leq k-1$ is odd then by Theorem $\ref{th1}$ and induction we have that 
$$\overline{\binom{2^{n+1}l+j}{k}}=\overline{\binom{2^{n}l+\lfloor j/2 \rfloor}{\lfloor k/2 \rfloor}}=0.$$
\end{proof}

\begin{lem}\label{lem2}
The map 
$$B^{(n)}_{k-d}\xrightarrow{\sigma_{d,n}}B^{(n)}_{k}$$ 
is surjective if and only if there exits $\alpha\in B^{(n)}_{k-d}$ such that $\alpha \sigma_{d,n}=x_1\cdots x_k$
\end{lem}

\begin{proof}
One way is trivial. In the other way suppose that there exits $\alpha\in B^{(n)}_{k-d}$ such that $\alpha \sigma_{d,n}=x_1\cdots x_k$. Note that the set $\{x_{i_1}\cdots x_{i_k} \mid 1\leq i_1<\cdots<i_k\leq n \}$ is a basis for the vector space $B^{(n)}_{k}$. By the natural action of the group of permutations of $n$-elements $\Sigma_n$ on $B^{(n)}$, given $g \in \Sigma_n$ we have that
\begin{align*}
x_{g(1)}\cdots x_{g(k)}&=g*(x_{1}\cdots x_{k})\\
&=g *( \alpha \sigma_{d,n})\\
&= (g * \alpha)(g * \sigma_{d,n})\\
&=(g * \alpha) \sigma_{d,n}
\end{align*}

\end{proof}
The next theorem is one of the key results of this section.
\begin{thm}\label{thmp}
Let $d=2^ml$. Then $\sigma_{d,n}$ is semi-regular for all $n=d+i$, with $0\leq i \leq 2^{m+1}-1$.
\end{thm}

\begin{proof}
For $0\leq i \leq 1$ the result follows from Theorem \ref{trivialsem}. Now notice that for all $n=d+i$ with $0\leq i \leq 2^{m+1}-1$ we have that $n<3d$. Thus, by Theorem $\ref{th2}$ we need to prove that for $D=\lceil (n+d)/2\rceil$ the map 
$$B^{(n)}_{D-1-d}\xrightarrow{\sigma_{d,n}}B^{(n)}_{D-1}$$ 
is injective and the map
$$B^{(n)}_{D-d}\xrightarrow{\sigma_{d,n}}B^{(n)}_D$$
is surjective. 
Let us prove by induction over $i$ that the result is true for all $2 \leq i \leq 2^{m+1}-1$. If $i=2$ we have that $n=d+2$ thus $D=d+1$ therefore we need to show that
$$B^{(d+2)}_{1}\xrightarrow{\sigma_{d,d+2}}B^{(d+2)}_{d+1}$$
is surjective. By equation (\ref{re1}) and Lemma \ref{lemmul} we have that
\begin{align*}
\sigma_{1,d+1}\sigma_{d,d+2}&=\sigma_{1,d+1}(\sigma_{d,d+1}+\sigma_{d-1,d+1}x_{d+2})\\
&=\overline{\binom{d+1}{1}}\sigma_{d+1,d+1}+\overline{\binom{d}{1}}\sigma_{d,d+1}x_{d+2}\\
&=\sigma_{d+1,d+1}=x_1\cdots x_{d+1}.
\end{align*}
Therefore by Lemma $\ref{lem2}$ the map is onto. Suppose now that the result is true for $d+i$ with $2\leq i \leq 2^{m+1}-2$. Let us see that $\sigma_{d,n}$ is semi-regular for $n=d+i+1$. Suppose that $i+1$ is even, so $4\leq i+1 \leq 2^{m+1}-2$ then taking $k=(i+1)/2$ we have that $2\leq k \leq 2^{m}-1$, $n=d+2k$ and
$$D= \left\lceil {\frac{d+i+1+d}{2}} \right\rceil= \left\lceil \frac{d+i+d}{2} \right\rceil=d+k.$$
By Theorem \ref{th2} we want to show that the map 
$$B^{(n)}_{D-1-d}\xrightarrow{\sigma_{d,n}}B^{(n)}_{D-1}$$ 
is injective and the map
$$B^{(n)}_{D-d}\xrightarrow{\sigma_{d,n}}B^{(n)}_{D}$$
is surjective.  By induction we have that $\sigma_{d,n-1}$ is semi-regular. Since $\left\lceil (n-1+i)/2 \right\rceil \\*=\left\lceil (d+i+d)/2 \right\rceil=D$ we have that the map 
$$B^{(n-1)}_{D-1-d}\xrightarrow{\sigma_{d,n-1}}B^{(n-1)}_{D-1}$$ 
is injective. By Lemma \ref{leminj} we have that $$B^{(n)}_{D-1-d}\xrightarrow{\sigma_{d,n}}B^{(n)}_{D-1}$$ 
is injective. Now, want to show that
$$B^{(n)}_{D-d}\xrightarrow{\sigma_{d,n}}B^{(n)}_{D}$$
is surjective. Since $D=d+k$, the we need to show that
$$B^{(n)}_{k}\xrightarrow{\sigma_{d,n}}B^{(n)}_{d+k}$$
is surjective. By equation (\ref{re1}) and Lemma \ref{lemmul} we have that
\begin{align*}
\sigma_{k,d+k}\sigma_{d,n}&= \sigma_{k,d+k}\sigma_{d,d+2k}\\
&=\sigma_{k,d+k}\sum_{j=0}^{d}\sigma_{d-j,d+k}\sigma_{j}(x_{d+k+1},...,x_{d+2k})\\
&=\sum_{j=0}^{d}\overline{\binom{d+k-j}{k}} \sigma_{d+k-j,d+k}\sigma_{j}(x_{d+k+1},...,x_{d+2k}).
\end{align*}
Since $2\leq k \leq 2^{m}-1$ by Lemma \ref{lemco} we have that 
$$\sigma_{k,d+k}\sigma_{d,n}=\sigma_{d+k,d+k}=x_1\cdots x_{d+k}.$$
By Lemma $\ref{lem2}$ the map is onto. Suppose now that $i+1$ is odd. Since $2\leq i \leq 2^{m+1}-2$ then taking $k=i/2+1$ we have that $2\leq k \leq 2^{m}$, $n=d+2k-1$ and
$$D= \left\lceil {\frac{d+i+1+d}{2}} \right\rceil= \left\lceil \frac{d+i+d}{2} \right\rceil+1=d+k.$$ 
We want to show that the map 
$$B^{(n)}_{D-1-d}\xrightarrow{\sigma_{d,n}}B^{(n)}_{D-1}$$ 
is injective and the map
$$B^{(n)}_{D-d}\xrightarrow{\sigma_{d,n}}B^{(n)}_{D}$$
is surjective. By induction we have that $\sigma_{d,n-1}$ is semi-regular. Since $\left\lceil (n-1+i)/2 \right\rceil \\*=\left\lceil (d+i+d)/2 \right\rceil=d+k-1=D-1$ we have that the map 
$$B^{(n-1)}_{D-2-d}\xrightarrow{\sigma_{d,n-1}}B^{(n-1)}_{D-2}$$ 
is injective and the map
$$B^{(n-1)}_{D-1-d}\xrightarrow{\sigma_{d,n-1}}B^{(n-1)}_{D-1}$$ 
is surjective. But notice that 
\begin{align*}
\dim B^{(n-1)}_{D-1-d}&=\dim B^{(n-1)}_{k-1}\\
&=\binom{n-1}{k-1}\\
&=\binom{d+i}{i/2}\\
&=\binom{d+i}{d+i/2}\\
&=\dim B^{(n-1)}_{d+k-1}\\
&=\dim B^{(n-1)}_{D-1}.
\end{align*}
Therefore the map $$B^{(n-1)}_{D-1-d}\xrightarrow{\sigma_{d,n-1}}B^{(n-1)}_{D-1}$$ 
is injective. By Lemma \ref{leminj} we have that $$B^{(n)}_{D-1-d}\xrightarrow{\sigma_{d,n}}B^{(n)}_{D-1}$$ 
is injective. Now, we want to show that
$$B^{(n)}_{D-d}\xrightarrow{\sigma_{d,n}}B^{(n)}_{D}$$
is surjective. Since $k=D-d$ then we want to show that
$$B^{(n)}_{k}\xrightarrow{\sigma_{d,n}}B^{(n)}_{d+k}$$
is surjective. By equation (\ref{re1}) and Lemma \ref{lemmul} we have that
\begin{align*}
\sigma_{k-1,d+k}\sigma_{d,n}&= \sigma_{k-1,d+k}\sigma_{d,d+2k-1}\\
&=\sigma_{k-1,d+k}\sum_{j=0}^{d}\sigma_{d-j,d+k}\sigma_{j}(x_{d+k+1},...,x_{d+2k-1})\\
&=\sum_{j=0}^{d}\overline{\binom{d+k-1-j}{k-1}} \sigma_{d+k-1-j,d+k}\sigma_{j}(x_{d+k+1},...,x_{d+2k-1}).
\end{align*}
Since $2\leq k \leq 2^m$, then $1\leq k-1 \leq 2^m-1$. Thus, by Lemma \ref{lemco} we have that 
$$\sigma_{k-1,d+k}\sigma_{d,n}=\sigma_{d+k-1,d+k}.$$
Therefore, $x_1\sigma_{k-1,d+k}\in B^{(n)}_k$ and $x_1\sigma_{k-1,d+k}\sigma_{d,n}=x_1\sigma_{d+k-1,d+k}=\sigma_{d+k,d+k}=x_1\cdots x_{d+k}$. By Lemma $\ref{lem2}$ the map is onto.
\end{proof}

\begin{lem} Let $m$ be a positive integer
 \[
\overline{\binom{j}{2^m}}=
 \begin{cases}
     \hfill 0 \hfill & \text{if $j=2^{m+1}$}\\
     \hfill 1 \hfill & \text{if $2^m \leq j \leq 2^{m+1}-1$}
\end{cases}
\]
where $\overline{k}$ denotes the image of $k$ in $\mathbb{F}_2$.
\end{lem}

\begin{proof}
If $m=1$ we have that 
$$\overline{\binom{4}{2}}=0,\:\text{and}\:\overline{\binom{3}{2}}=\overline{\binom{2}{2}}=1.$$
Suppose the result is true for $m\geq  1$ let us prove that it is true for $m+1$. Note that by Theorem \ref{th1} and induction
$$\overline{\binom{2^{m+2}}{2^{m+1}}}=\overline{\binom{2^{m+1}}{2^{m}}}=0.$$
Now, if $2^{m+1}\leq j \leq 2^{m+2}-1$, then $2^{m}\leq \lfloor j/2 \rfloor \leq 2^{m+1}-1$ so by Theorem \ref{th1} and induction
$$\overline{\binom{j}{2^{m+1}}}=\overline{\binom{\lfloor j/2 \rfloor}{2^{m}}}=1.$$
\end{proof}

\begin{lem}\label{lemin2}
Let $d=2^m$ and $n=3d=2^{m+1}+2^m$. Then 
$$\sigma_d(x_{i_1},\dots,x_{i_{2^{m+1}}})\sigma_{d,n}=x_{i_1}\cdots x_{i_{2^{m+1}}}+\sigma_{2d,n}$$
\end{lem}
\begin{proof}
Let $\mu \in \Supp(\sigma_{2d,n})$. Let $j$ be the number of common variables between $\mu$ and $x_{i_1}\cdots x_{i_{2^{m+1}}}$. Note that $2^m \leq j \leq 2^{m+1}$. $\mu$ will occur once in the product \\* $\sigma_d(x_{i_1},\dots,x_{i_{2^{m+1}}})\sigma_{d,n}$ for each occurrence of a sub-monomial of $\mu$ of degree $d=2^m$ in $\sigma_d(x_{i_1},\dots,x_{i_{2^{m+1}}})$. There are exactly $\binom{j}{2^m}$ such sub-monomials. If $j=2^{m+1}$ then $\mu =x_{i_1}\cdots x_{i_{2^{m+1}}}$ and by the above lemma this element appears an even number of times. So $x_{i_1}\cdots x_{i_{2^{m+1}}}\not\in \Supp ( \sigma_d(x_{i_1},\dots,x_{i_{2^{m+1}}})\sigma_{d,n})$. If $2^m\leq j \leq 2^{m+1}-1$ then by above lemma this element appears an odd number of times. So $\mu \in \Supp ( \sigma_d(x_{i_1},\dots,x_{i_{2^{m+1}}})\sigma_{d,n})$. Thus,
 $$\sigma_d(x_{i_1},\dots,x_{i_{2^{m+1}}})\sigma_{d,n}=x_{i_1}\cdots x_{i_{2^{m+1}}}+\sigma_{2d,n}.$$
\end{proof}

\begin{lem}\label{lemso}
Let $d=2^m$ and $n=3d$. Then the map
$$B^{(n)}_{d+1}\xrightarrow{\sigma_d}B^{(n)}_{2d+1}$$
is surjective.
\end{lem}
\begin{proof}
First, let us see that 
$$\sigma_{d-1,2d+1}\sigma_{d,n}=\sigma_{2d-1,2d+1}.$$
By equation (\ref{re1}), Lemma \ref{lemmul} and Lemma \ref{lemco} we have
\begin{align*}
\sigma_{d-1,2d+1}\sigma_{d,n}&=\sigma_{k,d+k}\sum_{j=0}^{d}\sigma_{d-j,2d+1}\sigma_{j}(x_{2d+2},...,x_{3d})\\
&=\sum_{j=0}^{d}\overline{\binom{d+(d-1-j)}{d-1}} \sigma_{2d-1-j,2d+1}\sigma_{j}(x_{2d+2},...,x_{3d})\\
&=\sum_{j=0}^{d}\overline{\binom{2^m+(2^m-1-j)}{2^m-1}} \sigma_{2d-1-j,2d+1}\sigma_{j}(x_{2d+2},...,x_{3d})\\
&=\sigma_{2d-1,2d+1}.
\end{align*}
Note that
\begin{align*}
\sigma_{2d-1,2d+1}&=x_1\sigma_{2d-2}(x_2,...,x_{2d+1})+x_2\sigma_{2d-2}(x_1,x_3,...,x_{2d-2})\\
&+ x_1x_2\sigma_{2d-3}(x_3,...,x_{2d-2})+x_3\cdots x_{2d+1}.
\end{align*}
Therefore \
$$x_1x_2\sigma_{d-1,2d+1}\sigma_{d,n}=x_1x_2\sigma_{2d-1,2d+1}=x_1x_2x_3\cdots x_{2d+1}.$$
Since $x_1x_2\sigma_{d-1,2d+1}\in B^{(n)}_{d+1}$ then by Lemma $\ref{lem2}$ the map is onto.
\end{proof}

\begin{thm}\label{thmf}
Let $d=2^m$. If $n=d+2^{m+1}$ then $\sigma_{d,n}$ is semi-regular.
\end{thm}

\begin{proof}
Suppose $n=d+2^{m+1}$ then $n=3d$. By Theorem \ref{th2} we want to show that the maps
$$B^{(n)}_{d-1}\xrightarrow{\sigma_{d,n}}B^{(n)}_{2d-1}$$ 
and
$$\frac{B^{(n)}_{d}}{\sigma_{d,n} B^{(n)}_{0}} \xrightarrow{\sigma_{d,n}}B^{(n)}_{2d}$$ 
are injective and the map
$$B^{(n)}_{d+1}\xrightarrow{\sigma_{d,n}}B^{(n)}_{2d+1}$$
is surjective. By Theorem \ref{thmp} we have that $\sigma_{d,n-1}$ is semi-regular. Since $n-1 \\* <n=3d$ then $\lceil (n-1+d)/2\rceil=2d$, thus the map 
$$B^{(n-1)}_{d-1}\xrightarrow{\sigma_{d,n-1}}B^{(n-1)}_{2d-1}$$ 
is injective. So, by Lemma \ref{leminj} the map 
$$B^{(n)}_{d-1}\xrightarrow{\sigma_{d,n}}B^{(n)}_{2d-1}$$ 
is injective. Also, by Lemma \ref{lemso} the map
$$B^{(n)}_{d+1}\xrightarrow{\sigma_{d,n}}B^{(n)}_{2d+1}$$
is surjective. Finally, by Lemma \ref{lemin2} we have that 
$$\sigma_d(x_{i_1},\dots,x_{i_{2d}})\sigma_{d,n}=x_{i_1}\cdots x_{i_{2d}}+\sigma_{2d,n}.$$
Consider the set
$$S=\{x_{i_1}\cdots x_{i_{2d}}+\sigma_{2d,n}\mid 1\leq i_{1}\leq \cdots \leq i_{2d} \leq n \}$$
Note that the set $S \setminus \{x_1\cdots x_{2^{m+1}}+\sigma_{2d,n}\}\subset B^{(n)}_{2d}$ is linearly independent. And this set has $\binom{n}{2d}-1=\dim B^{(n)}_{2d}-1$ elements. Therefore the map
$$\frac{B^{(n)}_{d}}{\sigma_d B^{(n)}_{0}} \xrightarrow{\sigma_{d,n}}B^{(n)}_{2d}$$ 
is injective. By all the above, $\sigma_{d,n}$ is semi-regular.
\end{proof}
The following theorem is our main result of this section.
\begin{thm}\label{gamsem2}
Let $d\geq 2$, where $d=2^{m}l$ with $l$ an odd number, and $m$ a non-negative integer. Then
\begin{enumerate}[(a)]
\item
If $l>1$, $\sigma_{d,n}$ is semi-regular if and only if $n=d,d+1,...,d+2^{m+1}-1$.
\item
If $l=1$, $\sigma_{d,n}$ is semi-regular if and only if $n=d,d+1,...,d+2^{m+1}$.
\end{enumerate}
\end{thm}

\begin{proof}
Suppose first that $m=0$. Thus, $d$ is an odd number. By Theorem \ref{gamsem1} we have that (a) is true in the case $m=0$. Suppose now that $m\ge 1$. Suppose that $\sigma_{d,n}$ is semi-regular. By Theorem \ref{thmbound} we have that $n\leq 3d$. By Corollary \ref{corind}
$$\ind(\sigma_{d,n})\geq \left\lceil \frac{n+d}{2} \right\rceil \geq \frac{n+d}{2}.$$
Note that 
$$\sigma_{2^m,n}\sigma_{d,n}=\overline{\binom{2^ml+2^m}{2^m}}\sigma_{2^m+d,n}=\overline{\binom{l+1}{1}}\sigma_{2^m+d,n}=0.$$
Thus $\ffd(\sigma_{d,n})\leq d+2^m$. Since $\sigma_{d,n}$ is semi-regular we have that $\ind(\sigma_{d,n})\leq \ffd(\sigma_{d,n})$. So $n+d\leq 2d+2^{m+1}$ from which we obtain $n\leq d+2^{m+1}$.
Suppose $l=1$, by Theorem \ref{thmp} and Theorem \ref{thmf} we have that $\sigma_{d,n}$ is semi-regular for $n=d,d+1,...,d+2^{m+1}$. So (b) is proved. Suppose $l>1$. We already know that $n\leq d+2^{m+1}$. Suppose $n=d+2^{m+1}=2^{m+1}l+2^{m+1}$ note that since $l>1$ then $n<3d$. Thus, if $\sigma_{d,n}$ is semi-regular then $\ind(\sigma_{d,n})=\left\lceil (n+d)/2 \right\rceil=d+2^m$. By Theorem \ref{th2} the map
$$B^{(n)}_{2^m}\xrightarrow{\sigma_{d,n}}B^{(n)}_{d+2^m}$$
is surjective. However,
\begin{align*}
\dim B^{(n)}_{2^m}=\binom{n}{2^m}&=\binom{d+2^{m+1}}{2^m}\\
&=\binom{d+2^{m+1}}{d+2^m}\\
&=\dim B^{(n)}_{d+2^m}.
\end{align*} 
So the map 
$$B^{(n)}_{2^m}\xrightarrow{\sigma_{d,n}}B^{(n)}_{d+2^m}$$
is injective. But that is not possible since $\sigma_{2^m,n}\sigma_{d,m}=0$. Hence we must have $n\leq d+2^{m+1}-1$. Conversely, by Theorem \ref{thmp} $\sigma_{d,n}$ is semi-regular for $n=d,d+1,...,d+2^{m+1}-1$. It proves (a).
\end{proof}

The following table gives a visual interpretation of Theorem \ref{gamsem2}.

\begin{table}[h]
$$
\begin{array}{c|c|c|c|c|c|c|c|c|c|c|c|c|c|}
n\backslash d&2&3&4&5&6&7&8&9&10&11&12&13&14\\
\hline
2&\texttt{x}&&&&&&&&&&&&\\
3&\texttt{x}&\texttt{x}&&&&&&&&&&&\\
4&\texttt{x}&\texttt{x}&\texttt{x}&&&&&&&&&&\\
5&\texttt{x}&&\texttt{x}&\texttt{x}&&&&&&&&&\\
6&\texttt{x}&&\texttt{x}&\texttt{x}&\texttt{x}&&&&&&&&\\
7&&&\texttt{x}&&\texttt{x}&\texttt{x}&&&&&&&\\
8&&&\texttt{x}&&\texttt{x}&\texttt{x}&\texttt{x}&&&&&&\\
9&&&\texttt{x}&&\texttt{x}&&\texttt{x}&\texttt{x}&&&&&\\
10&&&\texttt{x}&&&&\texttt{x}&\texttt{x}&\texttt{x}&&&&\\
11&&&\texttt{x}&&&&\texttt{x}&&\texttt{x}&\texttt{x}&&&\\
12&&&\texttt{x}&&&&\texttt{x}&&\texttt{x}&\texttt{x}&\texttt{x}&&\\
13&&&&&&&\texttt{x}&&\texttt{x}&&\texttt{x}&\texttt{x}&\\
14&&&&&&&\texttt{x}&&&&\texttt{x}&\texttt{x}&\texttt{x}\\
\end{array}
$$
\caption{Semi-Regularity of $\sigma_{d,n}$. The values when $\sigma_{d,n}$ is semi-regular are marked with an \texttt{x}}
\end{table}

\section{Most homogeneous sequences are semi-regular} \label{conj1}

In her thesis \cite[\S 3.1]{b} Bardet states ``Nous conjecturons tout de m\`{e}me qu'une suite `tir\'{e} au hazard' sera semi-r\'{e}guli\`{e}re sur $\gft$, dans le sens ou la proportion de suites semi-r\'{e}guli\`{e}res tend vers 1 quand $n$ tend vers l'infini'' (We conjecture none the less that a sequence `chosen at random' will be $\gft$-semi-regular in the sense that the proportion of sequences that are semi-regular tends to 1 as $n$ tends to infinity). We prove this result in its broadest interpretation: namely that the proportion of homogeneous subsets of $\Bn$ that are  semi-regular tends to 1 as $n$ tends to infinity.

\begin{lem} Let $k$ be a positive integer and suppose that $\{ \lambda_1,\dots, \lambda_m\}$ spans $B_k$. Then $ \lambda_1,\dots, \lambda_m$ is a semi-regular sequence.
\end{lem}

\begin{proof} Notice that $\deg(\lambda_i) = k$ so $(\lambda_1, \dots, \lambda_m)\cap B_j = \{0\}$ for all $j <k$. On the other hand, $(\lambda_1, \dots, \lambda_m)\cap B_k = B_k$, so $\ind(\lambda_1, \dots, \lambda_m) = k$. For any homogeneous $f \in B$ we have $\deg(f) + \deg(\lambda_i) \geq k$ so $k$-semi-regularity is trivially satisfied.
\end{proof}

\begin{lem}  Let $\{ \lambda_1,\dots, \lambda_m\}$  be a semi-regular sequence of homogeneous elements of $B$ and let $D =
\ind ( \lambda_1,\dots, \lambda_m)$ . If $\nu \in B$ is a homogeneous element of degree greater than or equal to $D$ then $ \lambda_1,\dots, \lambda_m, \nu$ is also semi-regular.
\end{lem}

\begin{proof} Since $\deg(\nu) \geq D$ we have 
$$ ( \lambda_1,\dots, \lambda_m) = ( \lambda_1,\dots, \lambda_m, \nu)
$$
so the degree of regularity does not change and again we see for any homogeneous $f \in B$ that $\deg(f) + \deg(\nu) \geq D$. So $D$-semi-regularity is again trivially verified.
\end{proof}

\begin{prop} \label{spansets} Let $V_n$ be an $n$-dimensional $\gft$-vector space and let $\mcP(V_n\backslash \{0\})$ be the set of subsets of $V_n\backslash \{0\}$. Let $S_n$ be the number of subsets of $V_n\backslash \{0\}$ that span $V_n$. Then
$$
\lim_{n\to \infty} \frac{S_n}{|\mcP(V_n\backslash \{0\})|} =1
$$
Equivalently the probability of randomly picked subset of $V_n$ being a spanning set goes to $1$ as $n\to \infty$.
\end{prop}

\begin{proof} We find an upper bound on the probability that a randomly picked subset $A$ does not span $V_n$. Note
that if a subset does not span $V_n$ then it must be contained in some $n-1$ dimensional subspace of $V_n$.
The probability that $A$ is contained in a particular $n-1$ dimensional subspace is $2^{2^{n-1}}/2^{2^{n}}=1/ 2^{2^{n-1}}$. There are $2^{n}-1$ such
subspaces, summing over all subsets we have that the probability is less than $(2^{n}-1)/2^{2^{n-1}}$. Taking a limit we have
$$
\lim_{n\to \infty} \frac{2^{n}-1}{2^{2^{n-1}}}=0
$$
as desired.
\end{proof}

\begin{thm} \label{srgen} Let $d$ be a positive integer and let $P_d(n)$ be the proportion of homogeneous sequences of degree greater than or equal to $d$ in $B^{(n)}$ that are  semi-regular. Then  
$$
\lim_{n\to \infty} P_d(n) =1
$$
\end{thm}

\begin{proof} Let $h_d(n)$ be the number of homogeneous elements of $B=\grft$ of degree greater than or equal to $d$. Then there are $2^{h_d(n)}$ possible homogeneous sequences (note that we are ignoring order in the sequence).
From Proposition 3.2 and Proposition 3.3 we know that if we pick a spanning set
of $B_d$ and add any elements from $B_i$ where $i > d$ then this is a semi-regular sequence. The number of homogeneous elements of $B$ of degree greater than $d$ is 
$$h_{d+1}(n) = h_d(n)- (|B_d|-1).$$
Let $S_d(n)$ be the number of spanning sets of $B_d$. Then a lower
bound on the proportion of semi-regular sequences is given by
$$\frac{S_d(n)2^{h_{d+1}(n)}}{2^{h_d(n)}}= \frac{S_d(n)}{2^{|B_d|-1} }
$$
Since $d\geq 1$, $\dim B_d \to \infty$ as $n\to \infty$ and so Proposition \ref{spansets}, implies that 
$$
\lim_{n \to \infty} \frac{S_d(n)}{2^{|B_d|-1} } =1.
$$
\end{proof}

\section{Non-Existence of Semi-Regular Sequences over $\bbF_2$}

In this section we prove the following theorem.

\begin{thm}
Let $d_1, \dots,d_m$ be a sequence of integers with $d_i \geq 2$ for some $1\leq i \leq m$. Then there exists an $N$ such that for all $n\geq N$, there cannot be a semi-regular sequence $\lambda_1,\dots,\lambda_m$ of homogeneous polynomials of degrees  $d_1, \dots,d_m$.
\end{thm}
This theorem implies that Conjecture 2 in \cite{bfs2} is false. 

The idea of the proof is the following. For $\mathbf{d}=(d_1,\dots,d_m)$, we define the function
\begin{equation}\label{eq:eqtao}
\tau_{\mathbf{d}}(n)=\ind \frac{(1+z)^n}{\prod_{i=1}^m(1+z^{d_i})}.
\end{equation}
We show that this function is bounded below by a linear function $g(n)=rn+c$, with $r>1/2$. Suppose that for some $j$ we have that $d_j\geq2$. Since $r>1/2$ then there exists $N$ such that for all $n\geq N$
$$\tau_{\bfd}(n)>\frac{n}{2}+\frac{d_j}{2}+1.$$
Suppose $\lambda_1,\dots,\lambda_m$ is a semi-regular sequence of homogeneous polynomials of degrees $d_1,\dots,d_m$ in $B^{(n)}$, $n\geq N$. Thus, by Theorem \ref{bfsthm}, $\ind(\lambda_1,\dots,\lambda_m)=\tau_{\bfd}(n)>(n+d_j+2)/2$. Also, by Theorem \ref{ffdelement} we have that 
$$D_{\mathrm{ff}}(\lambda_j)\leq \frac{n+d_j+2}{2}.$$
Therefore, $D_{\mathrm{ff}}(\lambda_j)<\ind(\lambda_1,\dots,\lambda_m)$, but Theorem\ref{indffd2} tells us that this is not possible for a semi-regular sequence.

\begin{lem}
\label{la}
 Let $f:\mathbb{N}\rightarrow\mathbb{R}$ be a non-decreasing function. If there exist $n_0$, $N\in\mathbb{N}$, and $A\in \mathbb{R}$, such that for all $n\geq n_0$ we have
$$f(n+N)\geq f(n)+A$$
then there exists a constant $c$ such that 
$$f(n)\geq (A/N)n+c$$
for all natural number $n$.
\end{lem}

\begin{proof}
Consider the function $g(n)=f(n_0)+(A/N)(n-(n_0+N))$. Let us show that for all $n\geq n_0$ we have that $f(n)>g(n)$. Let $m\geq n_0$. Write $m-n_0=lN+b$, where $b$ is an integer $b<N$. By hypothesis we have that
$$f(m)=f(n_0+b+lN)\geq f(n_0+b)+lA.$$
Since $f$ is non-decreasing we have that $f(m)\geq f(n_0)+lA$. Now, 
\begin{align*}
g(m)&=g(n_0+b+lN)\\
&=f(n_0)+(A/N)(n_0+b+lN-n_0-N)\\
&=f(n_0)+A(l-1)+(A/N)b.
\end{align*}
But $b<N$, so $g(m)<f(n_0)+Al\leq f(m)$. Thus, for all $n\geq n_0$ we have that $f(n)>g(n)$. Note that $g$ is defined as $$g(n)=(A/N)n+k,$$
where $k=f(n_0)-(A/N)(n_0+N)$. So, for all $n\geq n_0$ we have that $f(n)>(A/N)n+k$. Since we have a finite number natural numbers less that $n_0$ then for an  appropriate choice of a constant $c$ we have that for all natural number $n$  $f(n)>(A/N)n+c$.  
\end{proof} 

\begin{lem}
For any $u$ between $0$ and $n$, and any $a_j\in \mathbb{R}$
\begin{equation*}
\sum_{j=0}^na_j\bn{n}{j}=\sum_{j=0}^{u-d}\gamma(n,j,d)(a_j+a_{j+d})+\sum_{j=u-d+1}^u\gamma(n,j,d)a_j+\sum_{j=u+1}^n\bn{n}{j}a_j
\end{equation*}
\end{lem}

\begin{proof}
Note that by definition of $\gamma(n,j,d)$, we have that 
$$\bn{n}{j}=\gamma(n,j,d)+\gamma(n,j-d,d).$$ Thus
\begin{align*} 
\sum_{j=0}^ua_j\bn{n}{j}&=\sum_{j=0}^u(\gamma(n,j,d)+\gamma(n,j-d,d))a_j
\\
&=\sum_{j=0}^u\gamma(n,j,d)a_j+\sum_{j=0}^u\gamma(n,j-d,d)a_j
\\
&=\sum_{j=0}^u\gamma(n,j,d)a_j+\sum_{j=0}^{u-d}\gamma(n,j,d)a_{j+d}
\\
&=\sum_{j=0}^{u-d}\gamma(n,j,d)(a_j+a_{j+d})+\sum_{j=u-d+1}^u\gamma(n,j,d)a_j
\end{align*}
\end{proof}

\begin{lem}\label{lemind}
Let $N$, $d$ be natural numbers. Let
\begin{equation*}
\beta(z)=\sum_{j=0}^{\infty}b_jz^j.
\end{equation*} 
Suppose that 
$$\ind \beta(z)\geq 1$$
and $$b_i+b_{i-d}\geq 0$$
for all 
$$0 \leq i \leq \ind \beta(z)+ \ind \frac{(1+z)^N}{1+z^d}-d-1.$$\\
Then
\begin{equation*}
\ind (1+z)^N \beta(z) \geq \ind \beta(z)+\ind \frac{(1+z)^N}{1+z^d}-d.
\end{equation*}
\end{lem}

\begin{proof}
Let 
\begin{equation*}
(1+z)^N \beta(z)=\sum c_iz^i
\end{equation*}
and let $l = \ind \beta(z)$ and $s = \ind (1+z)^N/(1+z^d)$.
Suppose that $l\geq 1$
and $$b_i+b_{i-d}\geq 0$$
for all $0 \leq i \leq l+s-d-1.$

We want to show that
$$\ind (1+z)^N \beta(z) \geq l+s-d.$$
That is, $c_i>0$ for $0\leq i \leq l+s-d-1$.
Clearly $\ind (1+z)^N\beta(z)\geq l$. It remains to show that $c_{l+i}>0$, for $i = 0, \dots, s-d-1$. For  $0\leq i \leq s-d-1$ we have by above lemma that
\begin{align*}
c_{l+i} &= \sum_{j=0}^{N} b_{l+i-j}\bn{N}{j} \\
&= \sum_{j=0}^{s-1-d} \gamma(N,j,d) (b_{l+i-j} +b_{l+i-j-d}) \\
&+\sum_{j=s-d}^{s-1} \gamma(N,j,d) b_{l+i-j} + \sum_{j=s}^N \bn{N}{j} b_{l+i-j}.
\end{align*}
For $j= 0, \dots, s-1$, we have that $\gamma(N,j,d) > 0$, since $s = \ind (1+z)^N/(1+z^d)$. Also, $b_{l+i-j} +b_{l+i-j-d}\geq 0$, since $l+i-j \leq l+s-d-1$. So 
\begin{equation*}
\sum_{j=0}^{s-1-d} \gamma(N,j,d) (b_{l+i-j} +b_{l+i-j-d})\geq 0. 
\end{equation*}
Finally, if $j \geq s-d$, then $l + i-j \leq l + (s-d-1) -(s-d)=l-1$, so $%
b_{l+i-j}>0$. Hence 
\begin{equation*}
\sum_{j=s-d}^{s-1} \gamma(N,j,d) b_{l+i-j} + \sum_{j=s}^N \bn{N}{j} b_{l+i-j}>0. 
\end{equation*}
Thus, we have shown that $c_{i} >0$ for $0\leq i\leq l+s-d-1$. So 
\begin{equation*}
\ind (1+z)^N\beta(z) \geq l+s-d.
\end{equation*}
\end{proof}

\begin{thm}
\label{indlem} If $\ind \alpha(z)\geq 1$ and
\begin{equation*}
\ind \alpha(z) \geq \ind \frac{\alpha(z)}{1+z^d} +\ind \frac{(1+z)^N}{1+z^d}%
-d
\end{equation*}
then 
\begin{equation*}
\ind \frac{(1+z)^N\alpha(z)}{1+z^d} \geq \ind \frac{\alpha(z)}{1+z^d} +\ind 
\frac{(1+z)^N}{1+z^d}-d 
\end{equation*}
\end{thm}

\begin{proof}
Let 
\begin{equation*}
\alpha(z) = \sum a_iz^i, \quad   \frac{\alpha(z)}{1+z^d} = \sum b_i z^i=\beta(z)
\end{equation*}
and let $l = \ind \beta(z)$ and $s = \ind (1+z)^N/(1+z^d)$.
Suppose that
\begin{equation*}
\ind \alpha(z) \geq \ind \frac{\alpha(z)}{1+z^d} +\ind \frac{(1+z)^N}{1+z^d}%
-d
\end{equation*}
with $\ind \alpha(z)\geq 1$. In other words, $\ind \alpha(z)\geq 1$ and
\begin{equation*}
\ind \alpha(z) \geq l+s-d.
\end{equation*}
We want to show that
$$\ind (1+z)^N\beta(z) \geq l+s-d.$$
Since $\ind \alpha(z)\geq 1$ then $\ind \beta(z) \geq 1$. Also, note that
$$b_{i} +b_{i-d}= a_{i}>0$$ for all $0\leq i \leq l+s-d -1$, since $\ind \alpha(z) \geq l+s-d$. Thus, by Lemma \ref{lemind} we have that 
$$\ind (1+z)^N\beta(z) \geq l+s-d.$$
\end{proof}
\begin{lem}
\label{lb}
Let $d$, $N$ be natural numbers. Then 
\begin{equation*}
\ind \frac{(1+z)^{N}(1+z)^n}{1+z^d}\geq \ind \frac{(1+z)^n}{1+z^d}+\ind \frac{(1+z)^{N}}{1+z^d}-d
\end{equation*}
for all natural number $n$.
\end{lem}

\begin{proof}
Consider
\begin{equation*}
 \frac{(1+z)^n}{1+z^d} =\sum_{i=0}^{\infty} b_iz^i= \beta(z).
\end{equation*}
We want to show that
$$\ind (1+z)^N \beta(z) \geq \ind \beta(z)+\ind \frac{(1+z)^N}{1+z^d}-d.$$
Clearly $\ind \beta(z) \geq 1$. Also, we have that 
$$b_i+b_{i-d}=\bn{n}{i}\geq 0$$
for all $i$. Thus, the result follows from Lemma \ref{lemind}.
\end{proof}

\begin{thm}
\label{indlem4}
Suppose that
\begin{equation*}
r \in \left\{\frac{1}{n} \left( \ind \frac{(1+z)^n}{1+z^d}-d \right)| \: n \geq 1 \right\}
\end{equation*}
then there exists a $c$ such that 
\begin{equation*}
\ind \frac{(1+z)^n}{1+z^d} \geq rn+c
\end{equation*}
for all $n$.
\end{thm}

\begin{proof}
Let
\begin{equation*}
r = \frac{1}{N} \left( \ind \frac{(1+z)^{N}}{1+z^d} -d \right).
\end{equation*}
Consider the function 
\begin{equation*}
\tau_{(d)}(k)=\ind \frac{(1+z)^k}{1+z^d}.
\end{equation*}
By Lemma \ref{lb} we have
\begin{align*}
\tau_{(d)}(n+N)&=\ind \frac{(1+z)^n(1+z)^N}{1+z^d}
\\
&\geq \ind \frac{(1+z)^n}{1+z^d}+\ind \frac{(1+z)^{N}}{1+z^d}-d
\\
&=\tau_{(d)}(n)+(\tau_{(d)}(N)-d).
\end{align*}
By Lemma \ref{la} there exists $c$ such that 
$$\tau_{(d)}(n)\geq \frac{1}{N}(\tau_{(d)}(N)-d)n+c$$
for all $n$. In other words,
\begin{equation*}
\ind \frac{(1+z)^n}{1+z^d} \geq rn+c
\end{equation*}
for all $n$.
\end{proof}

\begin{thm}
\label{indlem2}
Let $\mathbf{d}=(d_1,\dots,d_m)$ and let $\mathbf{d'}=(d_1,\dots,d_m,d)$. Suppose that
\begin{equation*}
r=\frac{1}{N}\left( \ind \frac{(1+z)^{N}}{1+z^d}-d \right)
\end{equation*}
for some positive integer $N$. Consider the function $\tau_{\mathbf{d}}(n)$ as defined in (\ref{eq:eqtao}). If 
$$\tau_{\mathbf{d}}(n)\geq \tau_{\mathbf{d'}}(n)+rN$$
then
$$\tau_{\mathbf{d'}}(n+N)\geq \tau_{\mathbf{d'}}(n)+rN$$
\end{thm}

\begin{proof}
Consider 
\begin{equation*}
\alpha(z)= \frac{(1+z)^n}{\prod_{i=1}^m(1+z^{d_i})}.
\end{equation*}
In this case
\begin{align*}
&\tau_{\bfd}(n)=\ind \alpha(z)\\
&\tau_{\mathbf{d'}}(n)= \ind \frac{(1+z)^n}{\prod_{i=1}^m(1+z^{d_i})(1+z^d)}=\ind \frac{\alpha(z)}{1+z^d}\\
&\tau_{\mathbf{d'}}(n+N)= \ind \frac{(1+z)^{n+N}}{\prod_{i=1}^m(1+z^{d_i})(1+z^d)}=\ind \frac{\alpha(z)(1+z)^{N}}{1+z^d}.
\end{align*}
Also,
\begin{equation*}
rN=\ind \frac{(1+z)^{N}}{1+z^d}-d.
\end{equation*}
Suppose
$$\tau_{\mathbf{d}}(n)\geq \tau_{\mathbf{d'}}(n)+rN.$$
Thus,
\begin{align*}
\ind \alpha(z)\geq \ind \frac{\alpha(z)}{1+z^d}+ \ind \frac{(1+z)^{N}}{1+z^d}-d.
\end{align*}
By Theorem \ref{indlem}
\begin{align*}
\ind \frac{\alpha(z)(1+z)^{N}}{1+z^d}
\geq \ind \frac{\alpha(z)}{1+z^d}+ \ind \frac{(1+z)^{N}}{1+z^d}-d.
\end{align*}
Therefore,
$$\tau_{\mathbf{d'}}(n+N)\geq \tau_{\mathbf{d'}}(n)+rN.$$
\end{proof}

\begin{thm}
\label{keylem}
Let $\mathbf{d}=(d_1,\dots,d_m)$ and let $\mathbf{d'}=(d_1,\dots,d_m,d)$. Suppose that
\begin{equation*}
s=\frac{1}{N}\left( \ind \frac{(1+z)^{N}}{1+z^d}-d \right)
\end{equation*}
for some positive integer $N$. Suppose that there exist $r\geq s$ and $c$ such that 
$$\tau_{\mathbf{d}}(n)\geq rn+c,$$for all $n$. Then there exists $c'$ such that 
$$\tau_{\mathbf{d'}}(n)\geq sn+c',$$ for all $n$. 
\end{thm}

\begin{proof}
Let $\mathbf{d}=(d_1,\dots,d_m)$ and let $\mathbf{d'}=(d_1,\dots,d_m,d)$. Suppose that
\begin{equation*}
s=\frac{1}{N}\left( \ind \frac{(1+z)^{N}}{1+z^d}-d \right)
\end{equation*}
for some positive integer $N$. Suppose that there exist $r\geq s$ and $c$ such that 
$$\tau_{\mathbf{d}}(n)\geq rn+c,$$ for all $n$. 
Let us prove that for $c'=\min\{c-2sN,-sN,0\}$ we have 
$$\tau_{\mathbf{d'}}(n)\geq sn+c',$$ for all $n$.
If $s\leq 0$, the Theorem is true since $c'\geq0$ and $\tau_{\bfd'}(n)\geq 0$.
Suppose that $s>0$. Let $n$ be any natural number. We want to show that 
$$\tau_{\mathbf{d'}}(n)\geq sn+c'.$$
Let $k$ be the largest positive integer less than or equal to $n$ such that $\tau_{\mathbf{d'}}(k)\geq sk+(c-sN)$, and set $n_1=k$. If no such positive integer exists, set $n_1=0$. If $n_1=n$, the assertion is true so assume that $n_1<n$. Write $n-n_1-1=hN+b$ where $b$ is an integer $b<N$. Let $m=n_1+1+jN$ where $0\leq j\leq h$. Then
$$\tau_{\mathbf{d'}}(m)<sm+(c-sN)\leq rm+(c-sN).$$
Hence $\tau_{\mathbf{d}}(m)\geq \tau_{\mathbf{d'}}(m)+sN$. By Theorem $\ref{indlem2}$ we have that $\tau_{\mathbf{d'}}(m+N)\geq \tau_{\mathbf{d'}}(m)+sN$. So by iterating this argument,
$$\tau_{\mathbf{d'}}(m)\geq \tau_{\mathbf{d'}}(n_1+1)+jsN\geq \tau_{\mathbf{d'}}(n_1)+jsN$$
Hence,
\begin{align*}
\tau_{\mathbf{d'}}(n)&\geq \tau_{\mathbf{d'}}(n_1+1+hN)
\\
&\geq \tau_{\mathbf{d'}}(n_1)+hsN.
\end{align*}
If $n_1=k$ we have that
\begin{align*}
\tau_{\mathbf{d'}}(n)&\geq \tau_{\mathbf{d'}}(n_1)+hsN
\\
& \geq s(n_1)+(c-sN)+hsN
\\
& = s(n_1+hN)+(c-sN)
\\
&=s(n-1-b)+(c-sN)
\\
&=sn-s(1+b)+(c-sN)
\\
&\geq sn-sN+(c-sN)
\\
& \geq sn+c'.
\end{align*}
If $n_1=0$ we have
\begin{align*}
\tau_{\mathbf{d'}}(n)&\geq \tau_{\mathbf{d'}}(n_1)+hsN
\\
& \geq hsN
\\
&=s(n-1-b)
\\
&=sn-s(1+b)
\\
& \geq sn-sN
\\
& \geq sn+c'.
\end{align*}
\end{proof}

\begin{thm}
\label{indlem3}
Let $\mathbf{d}=(d_1,\dots,d_m)$. Suppose that $r$ is such that for all $i$ there exits an $n_i$ such that
$$r\leq \frac{1}{n_i} \left(\ind \frac{(1+z)^{n_i}}{1+z^{d_i}}-d_i \right).$$
Then there exists a $c$ such that 
$$\tau_{\mathbf{d}}(n)=\ind \frac{(1+z)^n}{\prod_{i=1}^{m}(1+z^{d_i})}\geq rn+c$$
for all $n$.
\end{thm}

\begin{proof}
Let
\begin{equation*}
r_i=\frac{1}{n_i}\left( \ind \frac{(1+z)^{n_i}}{1+z^{d_i}}-d_i\right).
\end{equation*}
Reordering we can suppose that $r_1\geq r_2\geq\cdots \geq r_m \geq r$. By Theorem $\ref{indlem4}$ we have that there exists $c_1$ such that
$$\tau_{(d_1)}(n) \geq r_1n+c_1,$$ for all $n$.
By Theorem $\ref{keylem}$ we have that there exists $c_2$ such that
$$\tau_{(d_1,d_2)}(n) \geq r_2n+c_2,$$ for all $n$.
By iterating this argument we have that there exists $c$ such that
$$\tau_{\mathbf{d}}(n) \geq rn+c,$$ for all $n$.
\end{proof}

\begin{lem}
\label{lembino}
Let $d$ be a natural number. Then there exists $M$ such that for all $n\geq M$
$$\bn{2n}{n+d-j}-\bn{2n}{n-j}+\bn{2n}{n-j-d}-\bn{2n}{n-j-2d}>0$$ for all $0\leq j \leq d-\lfloor{d/2}\rfloor-1.$
\end{lem}

\begin{proof}
Let $p=\lfloor{d/2}\rfloor$. Note that for $0\leq j \leq d-p-1$ 
\begin{align*}
&\bn{2n}{n+d-j}-\bn{2n}{n-j}+\bn{2n}{n-j-d}-\bn{2n}{n-j-2d}
\\
&\geq \bn{2n}{n+d}-\bn{2n}{n}+\bn{2n}{n-2d+p+1}-\bn{2n}{n-2d}. 
\end{align*}
Now,
\begin{align*}
&\bn{2n}{n+d}-\bn{2n}{n}+\bn{2n}{n-2d+p+1}-\bn{2n}{n-2d}
\\
&= \frac{(2n)!}{(n+d)!(n-d)!}-\frac{(2n)!}{n!n!}+\frac{(2n)!}{(n-2d+p+1)!(n+2d-p-1)!}
\\
&-\frac{(2n)!}{(n-2d)!(n+2d)!}.
\end{align*}
Note that 
\begin{align*}
&\frac{1}{(n+d)!(n-d)!}-\frac{1}{n!n!}+\frac{1}{(n-2d+p+1)!(n+2d-p-1)!}
\\
&-\frac{1}{(n-2d)!(n+2d)!}
\\
&=\frac{q(n)}{(n+2d-p-1)!(n+2d)!},
\end{align*}
where
\begin{align*}
q(n)&=\prod_{i=1}^{d}(n+d+i)\prod_{i=1}^{3d-p-1}(n-d+i)-\prod_{i=1}^{2d}(n+i)\prod_{i=1}^{2d-p-1}(n+i)
\\
&+\prod_{i=1}^{4d-p-1}(n-2d+p+1+i)-\prod_{i=1}^{4d-p-1}(n-2d+i).
\end{align*} 
Clearly $q(n)$ is a polynomial in $n$ of degree at most $4d-p-1$. The coefficient of $n^{4d-p-1}$ is easily seen to be zero and that of $n^{4d-p-2}$ is
\begin{align*}
&\left(d^2+\frac{d(d+1)}{2}+(3d-p-1)(-d)+\frac{(3d-p-1)(3d-p)}{2}\right)
\\
&-\left(\frac{(2d)(2d+1)}{2}+\frac{(2d-p-1)(2d-p)}{2}\right)
\\
&+\left((4d-p-1)(-2d+p+1)+\frac{(4d-p-1)(4d-p)}{2}\right)
\\
&-\left((4d-p-1)(-2d)+\frac{(4d-p-1)(4d-p)}{2}\right)
\\
&=4dp-d^2-p^2+4d-2p-1,
\end{align*}
which is positive for all $d\geq 1$. Thus the leading coefficient of $q(n)$ is positive and $q(n)$ is positive for all $n\geq M$, for some $M$.
\end{proof}

\begin{lem}\label{lemg}
Let $n$, $d$ be natural numbers. Then
\begin{equation*}
\gamma(2n,k,d)>0
\end{equation*}
for all $0\leq k \leq n+ \lfloor{d/2}\rfloor$.
\end{lem}

\begin{proof}
By definition we have
$$\gamma (2n,k,d)=\sum_{j=0}^{\lfloor k/d \rfloor}(-1)^j\bn{2n}{k-jd}.$$
We know that $\bn{2n}{j}$ is strictly increasing when $0\leq j \leq n$, therefore
\begin{equation*}
\gamma(2n,k,d)>0
\end{equation*}
for all $0\leq k \leq n$.
Now, for $n\leq k \leq n+ \lfloor {d/2} \rfloor$ we have
\begin{equation*}
\gamma(2n,k,d)=\bn{2n}{k}-\bn{2n}{k-d}+\gamma(2n,k-2d,d).
\end{equation*}
If $ k \leq n+ \lfloor {d/2} \rfloor$, then $k-2d\leq n$. Thus
$$\gamma(2n,k-2d,d)>0$$
for all $n\leq k \leq n+ \lfloor {d/2} \rfloor$. In order to finish let us show that
$$\bn{2n}{k}-\bn{2n}{k-d}\geq 0$$
for all $n\leq k \leq n+ \lfloor {d/2} \rfloor$. 
If $n\leq k \leq n+ \lfloor {d/2} \rfloor$
then
\begin{equation*}
n-d\leq k -d\leq n+ \lfloor {d/2} \rfloor-d \leq n-\lfloor{d/2}\rfloor
\end{equation*}
and
\begin{equation*}
n-\lfloor{d/2}\rfloor \leq 2n-k\leq n.
\end{equation*}
So, for $n\leq k \leq n+ \lfloor {d/2} \rfloor$ we have that
$$\bn{2n}{n-d}\leq \bn{2n}{k-d}\leq \bn{2n}{n-\lfloor{d/2}\rfloor}$$
and
$$\bn{2n}{n-\lfloor{d/2}\rfloor}\leq \bn{2n}{2n-k}\leq \bn{2n}{n}.$$
Thus, for $n\leq k \leq n+ \lfloor {d/2} \rfloor$  
\begin{align*}
\bn{2n}{k}-\bn{2n}{k-d}&=\bn{2n}{2n-k}-\bn{2n}{k-d}\\
&\geq \bn{2n}{n-\lfloor{d/2}\rfloor}-\bn{2n}{n-\lfloor{d/2}\rfloor}\\
&=0.
\end{align*}
The result is proved.
\end{proof}

\begin{thm}\label{thmind}
Let $d$ be a natural number. There exists $K$ such that for all $n\geq K$ we have
$$\ind \frac{(1+z)^{n}}{1+z^d}> \frac{n}{2}+d$$
\end{thm}

\begin{proof}
First let us prove that there exists $M$ such that for all $n\geq M$ we have
$$\ind \frac{(1+z)^{2n}}{1+z^d}> \frac{2n}{2}+d$$
By equation \eqref{eq:eq1}, we need to show that there exists $M$ such that for all $n\geq M$ we have
$$\gamma(2n,k,d)>0,$$ for all $0\leq k \leq n+d$. By Lemma \ref{lemg} we have that for any $n$
\begin{equation}\label{eq:eq2}
\gamma(2n,k,d)>0,\:\text{for all $0\leq k \leq n+ \lfloor{d/2}\rfloor$.}
\end{equation}
It remains to show that there exists $M$ such that for all $n\geq M$
$$\gamma(2n,n+d-j,d)>0$$
for all $0\leq j \leq d-\lfloor{d/2}\rfloor-1$. For all $0\leq j \leq d-\lfloor{d/2}\rfloor-1$, we have that 
\begin{align*}
\gamma(2n,n+d-j,d)&=\bn{2n}{n+d-j}-\bn{2n}{n-j}\\&+\bn{2n}{n-j-d}-\bn{2n}{n-j-2d}
\\
&+\gamma(2n,n-j-3d,d).
\end{align*}
By Lemma $\ref{lembino}$, we have that there exist $M$ such that for all $n\geq M$
$$\bn{2n}{n+d-j}-\bn{2n}{n-j}+\bn{2n}{n-j-d}-\bn{2n}{n-j-2d}>0$$
for all $0\leq j \leq d-\lfloor{d/2}\rfloor-1$. Also, by \eqref{eq:eq2} we have that
$$\gamma(2n,n-j-3d,d)>0,\:\text{for all $0\leq j \leq d-\lfloor{d/2}\rfloor-1$.}$$
Therefore, for all $n\geq M$
\begin{equation}\label{eq:eq3}
\gamma(2n,k,d)>0,\:\text{for all $n+\lfloor {d/2}\rfloor +1\leq k \leq n+d$.}
\end{equation}
Thus from \eqref{eq:eq2} and \eqref{eq:eq3} we have that for all $n\geq M$ 
$$\gamma(2n,k,d)>0,$$ for all $0\leq k \leq n+d$. In other words, for all $n\geq M$ 
\begin{equation}
\label{eq:in1}
\ind \frac{(1+z)^{2n}}{1+z^d}\geq \frac{2n}{2}+d+1>\frac{2n}{2}+d.
\end{equation}
So, for all $n\geq M$
\begin{equation}\label{eq:in2}
\ind \frac{(1+z)^{2n+1}}{1+z^d}\geq \ind \frac{(1+z)^{2n}}{1+z^d}\geq \frac{2n}{2}+d+1>\frac{2n+1}{2}+d.
\end{equation} 
From \eqref{eq:in1} and \eqref{eq:in2} we have that
$$\ind \frac{(1+z)^n}{1+z^d}>\frac{n}{2}+d$$
for all $n\geq K$, where $K=2M+1$.
\end{proof}
Now, let us prove the main theorem of this section.
\begin{thm}\label{nonexistance}
Let $\bfd=(d_1,\dots,d_m)$, with $d_j\geq 2$ for some $1 \leq j \leq m$. Then there exists an $N$ such that for all $n\geq N$, there are no semi-regular sequences of type $\bfd$ in $B^{(n)}$.
\end{thm}
\begin{proof}
By Theorem \ref{thmind}, for all $0\leq i \leq m$ there exists an $n_i$ such that
$$  \ind \frac{(1+z)^{n_i}}{1+z^{d_i}}> \frac{n_i}{2}+d_i.$$
Set
$$ r_i=\frac{1}{n_i}\left(\ind \frac{(1+z)^{n_i}}{1+z^{d_i}}-d_i \right)>\frac{1}{2}$$
and let $r=\min{r_i}$. By Theorem \ref{indlem3} there exists $c$ such that 
$$\tau_{\bfd}(n)\geq rn+c,\:\text{for all $n$}.$$
Suppose that for some $j$ we have that $d_j\geq2$. Since $r>1/2$, there exists $N$ such that for all $n\geq N$
$$\tau_{\bfd}(n)>\frac{n}{2}+\frac{d_j}{2}+1.$$
Suppose $\lambda_1,\dots,\lambda_m$ is a semi-regular sequence of homogeneous polynomials of degrees $d_1,\dots,d_m$ in $B^{(n)}$, $n\geq N$. Thus, by Theorem \ref{bfsthm}, $\ind(\lambda_1,\dots,\lambda_m)=\tau_{\bfd}(n)>(n+d_j+2)/2$. Since $d_j\geq2$, Theorem \ref{ffdelement} tells us that 
$$D_{\mathrm{ff}}(\lambda_j)\leq \frac{n+d_j+2}{2}.$$
This would imply that $D_{\mathrm{ff}}(\lambda_j)<\ind(\lambda_1,\dots,\lambda_m)$, but by Theorem \ref{indffd2} this is not possible for a semi-regular sequence.
\end{proof}

\section{Conclusion}

Since the introduction of the concept of a semi-regular sequence over $\gft$, it has been conjectured that such sequences are in some sense ``generic''. However little concrete progress has been made towards proving this conjecture. In fact even in one of the simplest and most important cases, that of quadratic sequences of length $n$ in $n$ variables, the question of the {\em existence} of semi-regular sequences for all $n$ remains open. In this paper we established four results about the existence of semi-regular sequences
\begin{enumerate}
  \item The proportion of  sequences of homogeneous elements in $n$ variables that are semi-regular tends to one as $n$ tends to infinity.
  \item A homogeneous element of degree $d$ can only be semi-regular if $n\leq 3d$.
  \item We established precisely when the symmetric element $$\sigma_{d,n} = \sum_{1\leq i_1<\cdots<i_d\leq n}x_{i_1}\cdots x_{i_d}$$ is semi-regular. In particular when $d=2^t$, $\sigma_{d,n}$ is semi-regular for all $d \leq n \leq 3d$.
  \item Sequences of a fixed length $m$ and fixed degree ${\bf d} =(d_1, \dots, d_m)$ are never semi-regular for sufficiently large $n$.
\end{enumerate}
While (1) is in a sense a statement that semi-regular sequences are generic, it doesn't imply that semi-regular sequences are `dense' in any way. For instance (4) suggests that for any $n$ that there is an $M$ such that there are no semi-regular sequences of length $m \leq M$. More importantly (1) gives us no information about special cases such as sequences of length $n$ in $n$ variables. What we would like to show is something like the following. There exists an $\epsilon$ such that if $m(n)=\lfloor \alpha n \rfloor + c$, then the proportion of sequences of length $m(n)$ in $n$ variables tends to one as $n$ tends to infinity whenever $\alpha >\epsilon$. This appears to be a hard problem. There do appear to be sporadic values of $(n,m)$ for which the proportion of semi-regular elements is low (such as $(n,m)=(10,12),(11,15)$ and $(15,14)$). These low proportions correspond precisely to values of $(n,m)$ for which the coefficient of $(1+z)^n/(1+z^2)^m$ is zero at the index. If this phenomenon can occur for arbitrarily large values of $n$ and $m$, then it is possible that Conjecture \ref{newconj} will be false.

\end{document}